\documentclass{amsart}

\usepackage[colorlinks,citecolor=niceblue,linkcolor=niceblue]{hyperref}
\usepackage{comment}
\usepackage{esint}
\usepackage{amssymb, amsrefs}
\usepackage{tikz}
\usetikzlibrary{calc}
\usetikzlibrary{patterns}
\usepackage{graphicx}
\usepackage{enumitem}
\usepackage{array}
\usepackage[export]{adjustbox}

\renewcommand{\MultipleCiteKeyWarning}[2]{}

\newtheorem{theorem}{Theorem}
\newtheorem{lemma}[theorem]{Lemma}
\newtheorem{corollary}[theorem]{Corollary}
\newtheorem{proposition}[theorem]{Proposition}
\theoremstyle{definition}

\newtheorem{remark}[theorem]{Remark}

\newtheorem{definition}[theorem]{Definition}

\newcommand{\eref}[1]{(\ref{e.#1})}
\newcommand{\tref}[1]{Theorem \ref{t.#1}}
\newcommand{\lref}[1]{Lemma \ref{l.#1}}
\newcommand{\pref}[1]{Proposition \ref{p.#1}}
\newcommand{\cref}[1]{Corollary \ref{c.#1}}
\newcommand{\fref}[1]{Figure \ref{f.#1}}
\newcommand{\sref}[1]{Section \ref{s.#1}}
\newcommand{\partref}[1]{\ref{part.#1}}
\newcommand{\dref}[1]{Definition \ref{d.#1}}
\newcommand{\rref}[1]{Remark \ref{r.#1}}

\newcommand{\deflink}[2]{\hyperref[#1]{#2}}

\newcommand{\EE}{\textup{(\deflink{d.energy_solution}{E})}}

\numberwithin{theorem}{section}
\numberwithin{equation}{section}

\newcommand{\R}{\mathbb{R}}

\newcommand{\grad}{\nabla}

\def\XXint#1#2#3{{\setbox0=\hbox{$#1{#2#3}{\int}$ }
\vcenter{\hbox{$#2#3$ }}\kern-.6\wd0}}

\newcommand{\ep}{\varepsilon}

\newcommand{\Diss}{\operatorname{Diss}}
\newcommand{\Dissbar}{\overline{\operatorname{Diss}}}

\newcommand{\one}{\mathbf{1}}

\definecolor{niceblue}{rgb}{0,0,0.7}
\definecolor{darkgreen}{rgb}{0,0.4,0}

\def\strikethrough#1{\setbox0\hbox{#1}\rlap{#1}\hbox to \wd0{\hss\strikebox\hss}}
\def\strikebox{\vrule height 0.6\ht0 depth -0.4\ht0 width 1.1\wd0}

\begin{document}

\title{On the geometry of rate-independent droplet evolution}
\author[W. M. Feldman]{William M Feldman}
\address{Department of Mathematics, University of Utah, Salt Lake City, Utah, 84112, USA}
\email{feldman@math.utah.edu}
\author[I. C. Kim]{Inwon C Kim} 
\address{Department of Mathematics, University of California, Los Angeles, California, 90095, USA}
\email{ikim@math.ucla.edu}
\author[N. Po\v{z}\'{a}r]{Norbert Po\v{z}\'{a}r}
\address{Faculty of Mathematics and Physics, Institute of Science and Engineering, Kanazawa University, Kakuma, Kanazawa 920-1192, Japan}
\email{npozar@se.kanazawa-u.ac.jp}
\begin{abstract}
We introduce a toy model for rate-independent droplet motion on a surface with contact angle hysteresis based on the one-phase Bernoulli free boundary problem. We consider a notion of energy solutions and show existence by a minimizing movement scheme. The main result of the paper is on the PDE conditions satisfied by general energy solutions: we show that the solutions satisfy the dynamic contact angle condition $\mathcal{H}^{d-1}$-a.e. along the contact line at every time.

\end{abstract}

\maketitle

\setcounter{tocdepth}{1}
\tableofcontents

 \section{Introduction}

In this work we consider a toy model, based on the one-phase Bernoulli free boundary problem, for the rate-independent motion of droplets on a solid surface with the effects of contact angle hysteresis. 

It is well understood in the physics and engineering literature that capillary droplets on many solid surfaces experience a phenomenon known as \emph{contact angle hysteresis} \cite{EralMannetjeOh}. Rather than a single contact angle specified by the material properties, as would be predicted by the classical Young's law in capillarity theory \cite{KarimReview}, there is a range of stable contact angles.  Consequently droplets can ``stick" to the solid surface; the contact line (the curve separating the wet and dry regions) does not move under small applied forces although the free surface (the liquid-air interface) does move. The origin and appropriate modeling of contact angle hysteresis is the subject of much debate in the engineering literature \cite{EralMannetjeOh,KarimReview}. 

Similar phenomena occur in the context of two phase flow in a porous medium, called capillary pressure hysteresis in this context \cite{ZMSNHJ,DPW,SS}.  Typically the flow velocity in each phase is determined by Darcy's law and the phase interface moves with velocity proportional to a pressure differential. However, when capillarity forces at the interface, where individual grains of the medium are only partially wetted, are of the same order as the pressure differential, pinning can occur. The origin of this pinning is exactly the microscopic surface roughness of the matrix medium and the associated contact angle hysteresis. The phase interface only advances or recedes when the pressure differential exceeds a certain threshold. 

The present paper is on the qualitative features of a macroscopic model for these types of interface pinning phenomena.  Our model is inspired by the one proposed by DeSimone, Grunewald, and Otto \cite{DeSimoneGrunewaldOtto}, with theory developed by Alberti and DeSimone \cite{alberti2011}.  

Instead of the capillarity energy with slowly varying volume constraint, as considered in \cite{DeSimoneGrunewaldOtto,alberti2011}, we study a quasi-static and rate-independent evolution associated with the Alt-Caffarelli one-phase energy functional under a slowly varying Dirichlet forcing. Although the Dirichlet driven problem does not conserve volume, it arises when studying the local behavior near the contact line where volume can be lost and gained from the bulk. Such scenario arises, for instance, in the low velocity {\it Wilhelmy plate method} for measuring dynamic contact angles, see for example \cite{SauerCarney}. There a plate or fiber is lifted slowly from a liquid bath and, in the reference frame of the plate, the surface height of the bath provides a time varying Dirichlet condition. While we expect to be able to handle the capillary energy instead of our linearized energy with mostly technical modifications, it is less clear how to address the volume constraint instead of the Dirichlet forcing.   See further discussion on the prescribed volume constraint problem in \sref{discussion}.

Our goal in this paper is to describe the motion law of solutions within an energetic framework. More precisely we introduce an energetic notion of solutions, \emph{global energetic solutions} (which we just refer to as ``energy solutions"). Energy solutions can be defined in a quite general setting, and we show that they satisfy the \emph{dynamic slope condition} \eqref{e.dynamic-slope-condition-intro}---which governs the continuous motion---pointwise almost everywhere along the contact line.  To our knowledge this is the first such rigorous result for a general class of weak solutions to a rate-independent model of droplet evolution. We are not aware of any prior works addressing the optimal space-time regularity of energy solutions in rate-independent systems. 
Besides its own interest, our result can be viewed as the first step toward understanding further regularity properties of energy solutions. In order to achieve  regularity properties for energy solutions, or to assess the existence of irregular energy solutions, it is necessary to first understand in detail the dynamic PDE conditions satisfied by solutions.  
For example, in our companion work \cite{FKPii}, we achieved optimal space-time regularity of a notion of solutions in a certain geometry which satisfies the dynamic slope condition in pointwise sense.

\subsection{ The model}\label{s.intro-problem-details} Consider a connected domain $U$ in $\R^d$ with compact complement.  For $u: U \to [0,\infty)$ consider the Alt--Caffarelli one-phase free boundary energy functional
\begin{equation}\label{e.energy-intro}
\mathcal{J}(u) = \int_U |\nabla u|^2 + {\bf 1}_{\{u>0\}} \ dx,
\end{equation}
where $\one_{\{u > 0\}}$ is the indicator function of the set $\Omega(u) := \{u>0\}$.

In the context of capillarity modelling the droplet's free surface (liquid--gas) is given by the graph of $u$ on $U$, and the positivity set $\Omega(u)$ is the wetted area: the contact between the liquid and solid phase; see \fref{one-phase}. The triple junction or contact line (liquid-gas-solid) is located at the free boundary $\partial\Omega(u) \cap U$. The Dirichlet energy $\int_U |\nabla u|^2 \,dx$ comes from a linearization of the surface area.

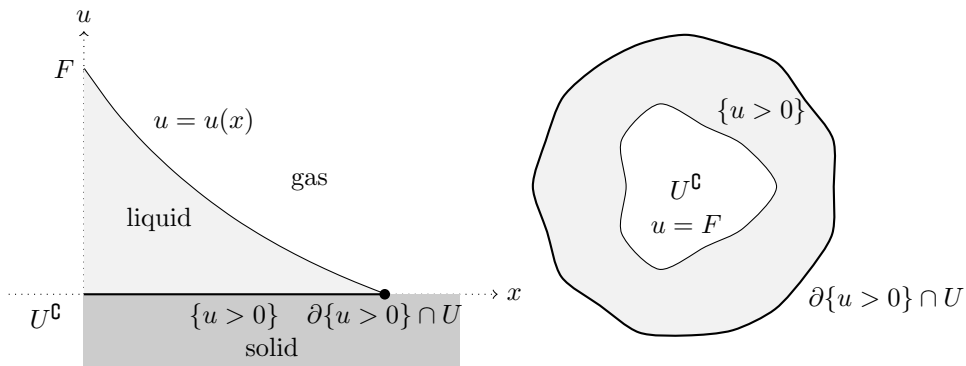
\begin{figure}
\centering
\begin{minipage}{0.48\textwidth}
\begin{tikzpicture}
\draw[->,dotted] (-1,0) -- (5.5,0) node[right] {$x$};
\draw[->,dotted] (0,0) -- (0,3.5) node[above] {$u$};
\draw[thick] plot [smooth] coordinates {
(0.0,3.0)
(0.4,2.45)
(0.8,2.0)
(1.2,1.62)
(1.6,1.29)
(2.0,1.0)
(2.4,0.75)
(2.8,0.529)
(3.2,0.333)
(3.6,0.158)
(4.0,0.0)
};
\fill[black!5!white] plot [smooth] coordinates {
(0.0,3.0)
(0.4,2.45)
(0.8,2.0)
(1.2,1.62)
(1.6,1.29)
(2.0,1.0)
(2.4,0.75)
(2.8,0.529)
(3.2,0.333)
(3.6,0.158)
(4.0,0.0)
} -- (0,0);
\fill[black!20!white] (0,0) -- (5,0) -- (5,-1) -- (0,-1);
\draw (1,1) node {liquid};
\draw (3,1.5) node {gas};
\draw (2.5,-0.7) node {solid};
\draw (2,0) node[below] {$\{u > 0\}$};
\draw (4,0) node[below] {$\partial \{u > 0\} \cap U$};
\fill (4,0) circle[radius=2pt];
\draw[thick] (0,0)--(4,0);
\draw (0,3) node[left] {$F$};
\draw (0.8,2.0) node[above right] {$u = u(x)$};
\draw (-0.5,0) node[below] {$U^\complement$};
\end{tikzpicture}
\end{minipage}
\hfill
\begin{minipage}{0.45\textwidth}
\begin{tikzpicture}
\draw[thick,fill=black!5!white] plot [smooth cycle] coordinates {
(1.96,0.0)
(1.95,0.633)
(1.58,1.15)
(1.2,1.65)
(0.607,1.87)
(0.0,2.02)
(-0.616,1.9)
(-1.17,1.61)
(-1.64,1.19)
(-1.87,0.607)
(-2.04,0.0)
(-1.85,-0.603)
(-1.66,-1.2)
(-1.15,-1.58)
(-0.629,-1.94)
(-0.0,-1.98)
(0.62,-1.91)
(1.18,-1.63)
(1.6,-1.16)
(1.94,-0.629)
};
\draw[fill=white] plot [smooth cycle] coordinates {
(1.2,0.0)
(0.759,0.551)
(0.259,0.797)
(-0.359,1.1)
(-0.859,0.624)
(-0.8,0.0)
(-0.859,-0.624)
(-0.359,-1.1)
(0.259,-0.797)
(0.759,-0.551)
}
;
\draw (1.5,-1.5) node[right] {$\partial \{u > 0\} \cap U$};
\draw (0,0) node {$U^\complement$};
\draw (0,-0.5) node {$u = F$};
\draw (1,1) node {$\{u > 0\}$};
\end{tikzpicture}
\end{minipage}
\caption{Side view (left) and the top view (right) of the setup for the one-phase free boundary problem.}
\label{f.one-phase}
\end{figure}

We recall that the Euler-Lagrange equation associated with the energy functional $\mathcal{J}$ is the one-phase Bernoulli free boundary problem
\begin{equation}\label{e.euler-lagrange-bernoulli}
\left\{
\begin{aligned}
\Delta u &= 0 && \hbox{in } \Omega(u) \cap U,\\
|\grad u|^2 &= 1 &&\hbox{on } \partial \Omega(u) \cap U.
\end{aligned}
\right.
\end{equation}

We are interested in  rate-independent evolutions associated with the energy functional $\mathcal{J}$, driven by a Dirichlet boundary data $F(t) : [0,T] \to (0,\infty)$ acting as the external forcing.
 We augment the energy functional \eref{energy-intro} by adding an energy dissipation in a form of a \emph{dissipation distance}: for any pair of sets $\Omega_0$ and $\Omega_1$ in $\R^d$ we define
 \[\Diss(\Omega_0,\Omega_1) = \mu_+|\Omega_1 \setminus \Omega_0| + \mu_-|\Omega_0 \setminus \Omega_1|.\]  
 This non-symmetric distance measures the energy dissipated by the motion of the free interface under (monotone) motion of the positive phase from state $\Omega_0$ to state $\Omega_1$.  The coefficients $\mu_+>0$ and $\mu_- \in (0,1)$ can be viewed as the friction forces per unit length of the free interface, respectively for advancing and receding regimes. 
 
 In what follows we will abuse notation and also write $\Diss(u,v) = \Diss(\Omega(u),\Omega(v))$ for the dissipation distance between the positive phase of two profiles $u$ and $v$.

In the simple model we assume that free interface can move only if the ``force'' $|\nabla u|^2 - 1$ per unit length on the interface coming from the first variation of potential energy \eref{energy-intro} can overcome the static friction force $\mu_+$ or $\mu_-$, depending on whether the contact line advances or recedes. Furthermore, the scale at which the contact line moves is much faster than the scale on which we observe the state of the system, and therefore at each time $t$ the interface is assumed to be in an equilibrium configuration and cannot move. This state $u(t)$ can be characterized as a local minimizer of the augmented energy functional
\begin{align}
\label{e.E-def-intro}
\mathcal{E}(u, u') := \mathcal{J}(u') + \Diss(u, u'),
\end{align}
 that expresses the total of the potential energy $\mathcal{J}(u')$ in an alternative state $u'$ and the energy dissipated by the friction forces on the contact line required to reach the alternative state. 
 
 The formal first variation of the augmented energy functional $\mathcal E(u,\cdot)$ results in the pinned one-phase free boundary problem
\begin{equation}\label{e.stability-condition-intro}
\begin{cases}
\Delta u= 0 & \hbox{ in } \ \{u>0\} \cap U,\\
1-\mu_- \leq |\grad u|^2 \leq 1+\mu_+ & \hbox{on } \partial \{u>0\} \cap U.
\end{cases}
\end{equation}
Now the gradient is allowed to take a range of values on the interface: this is the manifestation of contact angle hysteresis in this model as ``slope hysteresis".

The state might change with $t$ due to the  Dirichlet boundary condition
\begin{align*}
u(t) = F(t) \qquad \text{on  } \partial U
\end{align*}
at each $t \geq 0$.
Since we can observe the state only in an equilibrium, the evolution is rate-independent in the sense that the path does not depend on how fast $F$ changes, that is, the time variable can be monotonically reparametrized yielding an equivalent evolution.

Varying $F(t)$ pulls up (or down) the profile $u(t)$ but the free boundary remains pinned as long as the gradient at the free boundary is within the pinning interval in \eref{stability-condition-intro}. Once the gradient saturates one of the endpoints in the pinning condition \eref{stability-condition-intro} somewhere, the interface needs to move. And, indeed, the free boundary \emph{only} advances or recedes when the gradient saturates the corresponding endpoint of the pinning interval.  This heuristic suggests the \emph{dynamic slope condition}
\begin{equation}\label{e.dynamic-slope-condition-intro}
|\grad u(t,x)|^2 = 1\pm \mu_\pm \quad \hbox{ if } \quad \pm V_n(t,x) >0  \quad \hbox{on } \  \partial \Omega(u(t)) \cap U
\end{equation}
where $V_n(t,x)$ is the outward normal velocity of $\Omega(u(t))$ at $x \in \partial \Omega(u(t))$, and $\nabla$ always denotes the spatial gradient. This condition is analogous to the dynamic contact angle condition which is studied in the physics of capillarity.

    \subsection{Main results}

The view of the evolution as a continuum of stable states of the energy functional $\mathcal{E}$ motivates our present notion of solutions of the rate-independent evolution of the Bernoulli functional. However, the local minimality is strengthened to a global one, as is common in the theory of rate-independent systems \cite{mielke2015book,alberti2011}. Additionally, an evolution between any two states at two times along the solution must be energetically allowed: the dissipated energy required to reach the second state from the first cannot be larger than the difference of potential energies of the two states together with the work done by the external forcing.
    
\begin{definition}\label{d.energy_solution}
A measurable $u : [0,T] \to H^1(U)$ is a \emph{energy solution} \EE{} of the quasi-static evolution problem driven by Dirichlet forcing $F$ if the following hold:
\begin{enumerate}
\item (\emph{Forcing})  For all $t \in [0,T]$
\[ u(t) = F(t) \ \hbox{ on } \partial U.\]
\item (\emph{Global stability}) The solution $u(t)\in H^1(U)$ and satisfies  for all $t \in [0,T]$:
\begin{equation*}
\label{e.stability}
\mathcal{J}(u(t)) \leq \mathcal{J}(u') + \Diss(u(t),u') \qquad \text{for all }u' \in u(t)+ H^1_0(U). 
\end{equation*}
\item (\emph{Energy dissipation inequality}) For every $0 \leq t_0 \leq t_1 \leq T$ it holds
\begin{equation}\label{e.diss-inequality}
 \mathcal{J}(u({t_0}))-\mathcal{J}(u({t_1}))  + \int_{t_0}^{t_1}2 \dot{F}(t)P(t) \ dt \geq \Diss(u(t_0),u(t_1)).
\end{equation}
   Here $P(t) = P(u(t))= \int_{\partial U} \frac{\partial u(t)}{\partial n} \ dS$ is an associated pressure. 
  \end{enumerate}
  \end{definition}

In some works, for example \cite{alberti2011}, energy solutions are required to satisfy an \emph{energy dissipation balance} condition instead of the inequality we use. That notion is equivalent to ours, see \rref{energy-dissipation-equality} later for more details.

The Euler-Lagrange equations for the global stability condition \eref{stability} is the pinned one-phase problem \eref{stability-condition-intro}.  Taking the time derivative in the energy dissipation balance shows that smooth energy solutions satisfy the dynamic slope condition \eref{dynamic-slope-condition-intro}; see \eref{energy-diss-eq} and \eref{formal-energy-argument-conclusion} below for details of the computation.  Justifying this formal computation for general energy solutions is quite difficult and is the content of \tref{main-2} below.

 \subsubsection*{Minimizing movements scheme} The most typical way to construct an energy solution \EE{} is based on a time-incremental or minimizing movements scheme. For $\delta > 0$ consider the time-discrete approximation scheme 
\begin{equation}\label{e.minimizing-movement-scheme}
u_{\delta}^k \in \mathop{\textup{argmin}} \left\{ \mathcal{J}(w) + \Diss(u^{k-1}_\delta, w): w\in F(k\delta)+  H^1_0(U)\right\}.
\end{equation}
Using piecewise constant interpolation define, for all $t \in [0,T]$,
 \begin{equation*}
 u_{\delta}(t):= u_{\delta}^k \ \hbox{ and } \ F_\delta(t) = F(k\delta) \ \hbox{ if } \ t\in [k\delta, (k+1)\delta).
 \end{equation*}
The time incremental scheme produces energy solutions via a compactness idea introduced by Mainik and Mielke \cite{MainikMielke}.
\begin{definition}\label{d.mm-solution}

  Say that $u : [0,T] \to H^1(U)$ measurable is a \emph{minimizing movements energy solution} if there is a sequence $\delta_k \to 0$ and $u_{\delta_k}(t)$ solving the scheme \eref{minimizing-movement-scheme} so that $u_{\delta_k}(t) \to u(t)$ in $L^2(U)$ for \emph{every} $t \in [0,T]$.
  \end{definition}

Our first main result is the existence of an energy solution via the minimizing movements scheme.
\begin{theorem}[see \sref{energy-soln-existence} for a proof]\label{t.existence}
The $u_\delta(t)$ generated by the minimizing movements scheme satisfy the following:
\begin{enumerate}[label = (\roman*)]
\item\label{part.existence-BV-bounds} The states $u_{\delta}(t)$ are uniformly Lipschitz,  $\chi_{\Omega_{\delta}(t)}$ and $\mathcal{J}(u_{\delta})$ are uniformly bounded in $BV([0,T]; L^1(U))$ and $BV([0,T]; \R)$ respectively.
\item\label{part.existence-helly}  There is a subsequence $\delta_k \to 0$ and $u : [0,T] \to F(t) + H^1_0(U)$ so that for \emph{every} $t \in [0,T]$
\[\|u_{\delta_k}(t)- u(t)\|_{L^\infty(U)} \to 0 \hbox{ and } \ d_H(\Omega(u_{\delta_k}(t)), \Omega(u(t))) \to 0 \ \hbox{ as} \ \delta_k \to 0.\]
\item\label{part.existence-solution-props} Any such subsequential limit $u(t)$ is an energy solution.
\end{enumerate}
\end{theorem}
By \tref{existence} any minimizing movements solution, \dref{mm-solution}, is indeed an energy solution \EE{}. On the other hand, it remains open whether all energy solutions are minimizing movements solutions.

 The second main result shows that general energy solutions still satisfy the dynamic slope condition \eqref{e.dynamic-slope-condition-intro} in a geometric measure theoretic sense.

\begin{theorem}[see \pref{lr-limits} and \tref{ae-viscosity-prop}]\label{t.main-2}
Suppose $u$ is an energy solution on $[0,T]$. Then
\begin{enumerate}[label = (\roman*)]
\item\label{part.main2-p0}  \emph{(Basic regularity properties)} The states $u(t)$ are uniformly Lipschitz and non-degenerate and  $\mathcal{H}^{d-1}(\partial\Omega(t))$ is uniformly bounded in time. Also $t \mapsto \Omega(u(t))$ is in $BV([0,T];L^1(\R))$ and $u(t)$ has left and right limits in uniform metric at every time, denoted $u_\ell(t)$ and $u_r(t)$.
\item\label{part.main2-p1}  \emph{(Upper and lower envelopes)} The upper and lower semicontinuous envelopes of $u$, called $u^*$ and $u_*$, are themselves energy solutions and actually $u^*(t) = \max\{u_\ell(t), u_r(t)\}$ and $u_*(t) = \min\{u_\ell(t) , u_r(t)\}$.
\item\label{part.main2-p2}  \emph{(Dynamic slope condition a.e.)} For all $t \in [0,T]$ the function $u(t)$ satisfies the stability condition \eref{stability-condition-intro} and satisfies (in terms of $u^*$ and $u_*$) the dynamic slope condition \eref{dynamic-slope-condition-intro} at $\mathcal{H}^{d-1}$-almost every point of its free boundary $\partial \Omega(u(t)) \cap U$.
\end{enumerate} 
\end{theorem}

The highlight of this result is \partref{main2-p2}, where we show that the energy dissipation balance condition is satisfied in a strong local sense. The regularity properties in \partref{main2-p0} follow from relatively standard energy arguments; we are recalling them here for context. The precise description of the semicontinuous envelopes \partref{main2-p1} is a somewhat unusual result in the context of literature on rate-independent motion. Our analysis of the energetic properties of the upper and lower envelopes in \partref{main2-p1}, see \pref{lr-limits} below, is more typical of the theory of discontinuous viscosity solutions \cite{usersguide} and plays an important role in the proof of the dynamic slope condition.

  On the other hand, just under the energy solution property, we are not aware of any results on the higher regularity of solutions or on stronger (local) notions of the energy dissipation balance condition. In particular we are not aware of any prior literature which assesses the local regularity of the solution implied by the space-time effect of the energy dissipation balance condition. In order to take advantage of the Euler-Lagrange equation to prove higher regularity, as we have done in \cite{FKPii}, it is necessary to first understand exactly in what sense the PDE conditions are satisfied.

\subsection{Open questions}\label{s.discussion} 

We end the introduction with  several open questions. 

{\it  Energetic notion in anisotropic setting}. The global notion of energy solutions which we use in this paper provides many mathematical conveniences and allows much of our analysis. In particular this notion relies on the existence of a globally defined dissipation distance.  In the isotropic case that we consider in this paper there is such a simple and natural distance function.  Such a quantity is not always available. For example, it is not clear how to define any global dissipation distance in the case of the anisotropic media which arise from periodic homogenization \cite{Feldman21,FeldmanSmart,Kim08,CaffarelliMellet2}.

{\it  Locality of jump laws}. It is well-known that global energy solutions can exhibit un-realistic jump discontinuities, i.e. jumps that would be inconsistent with viscous / rate-dependent approximations. They tend to jump ``as early as possible" to preserve the global energy minimality at each time; see \fref{jump-law} for illustration.
   In the companion paper \cite{FKPii} we will discuss a notion of \textit{obstacle solutions}, which jump ``as late and little as possible". This is generally regarded as a physically preferable jump law.  The obstacle solution dissipates the ``right" amount of energy on its jumps, but it is not clearer whether it yields an energetic notion of jump dissipation. It would be interesting to study the possible connection of  the obstacle solutions with the available energetic notion of solutions which jumps ``as late as possible" in other settings, for instance the notion of  the balanced viscosity notions (\cite{MielkeRossiSavare1} or \cite[Chapter 3.8.2]{mielke2015book}).  
\begin{figure}
\begin{tabular}{ll}
 \includegraphics[width=0.5\textwidth]{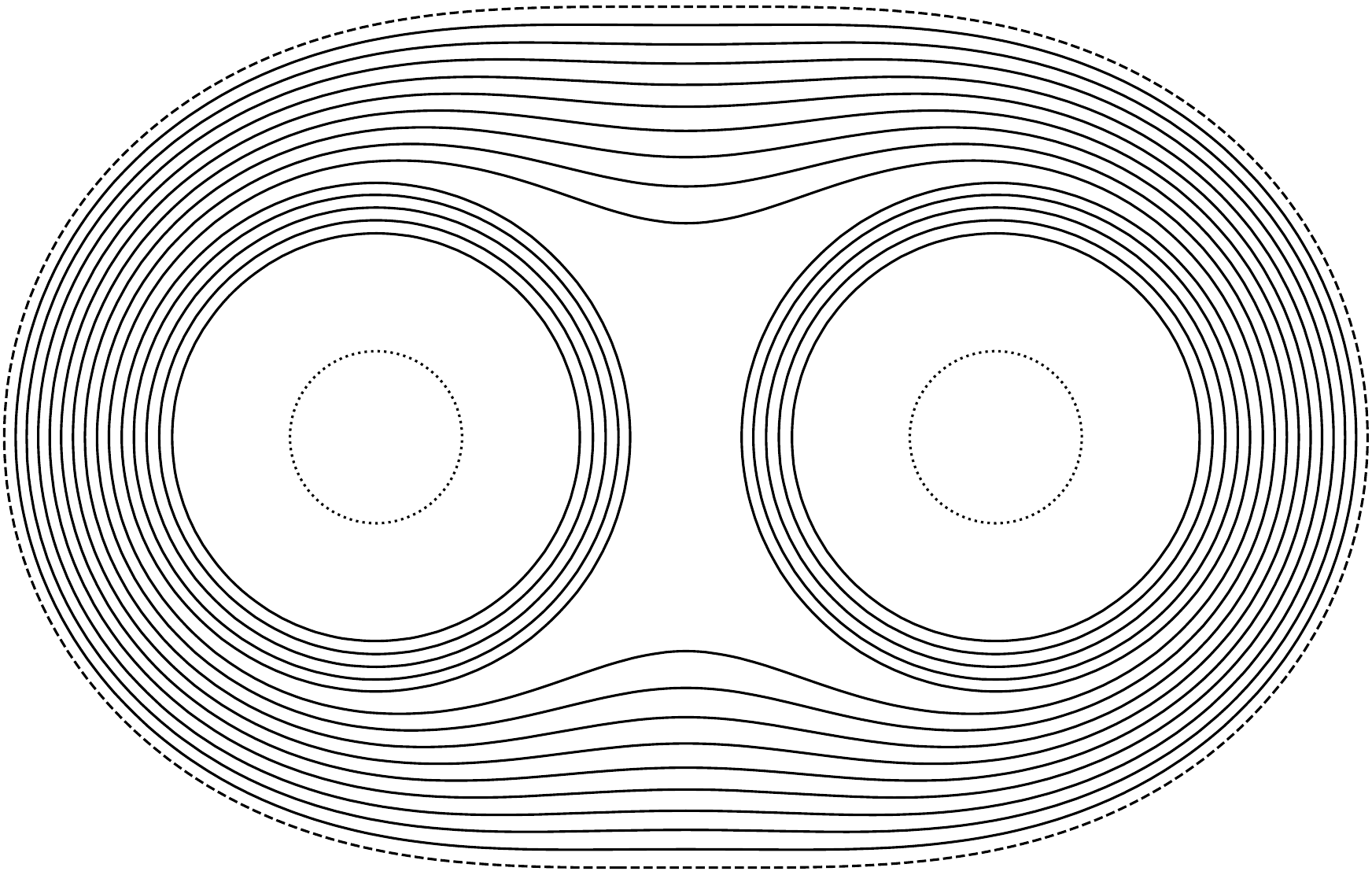}
 &
 \includegraphics[width=0.5\textwidth]{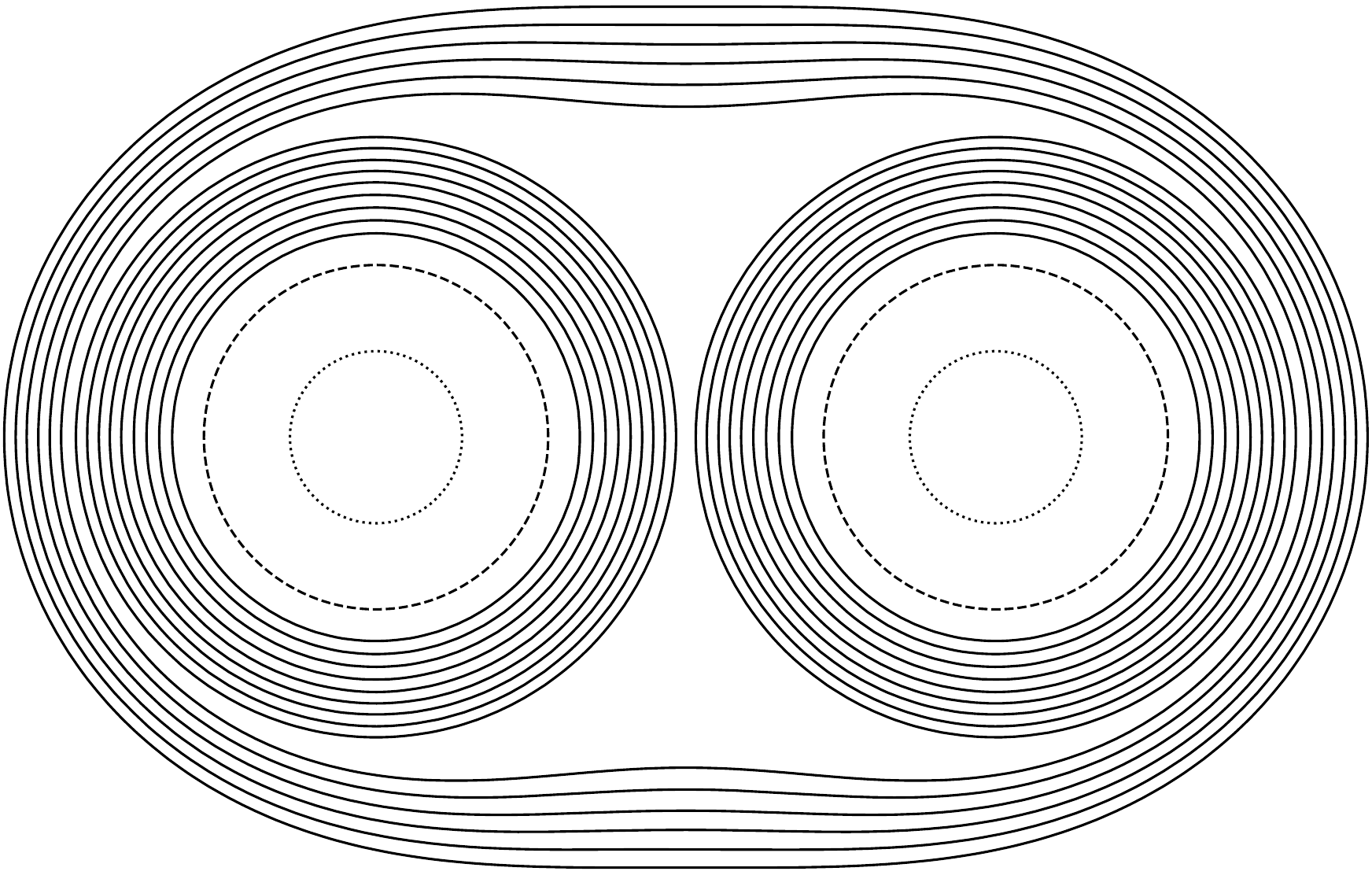}
\end{tabular}
\caption{Illustration of solutions jumping ``as early as possible'' (left) and ``as late as possible'' (right). An increasing Dirichlet boundary condition is prescribed on the two small dotted circles.}
\label{f.jump-law}
\end{figure}

{\it  Volume constraint.} Lastly, we hope that our approach can be extended to address the prescribed volume case, where we minimize \eqref{e.E-def-intro} with the constraint $\int u(t) \ dx = \textup{Vol}(t)$ at each time. The volume $\textup{Vol}(t)$ is a prescribed function of time driving the evolution, modeling evaporation/condensation processes. The first variation of this problem leads to the interior operator $-\Delta u = \lambda$  in the positive set of $u$, replacing the first equation in \eqref{e.stability-condition-intro}. Our theory can be applied rather easily to this case if the Lagrange multipler  $\lambda$ were  instead a priori given as a function of time.  The challenge with prescribing the volume $\textup{Vol}$  is the implicit dependence of $\lambda$ on both $\textup{Vol}$ and the geometry of the set $\{u>0\}$, which brings up various technical difficulties. For instance  the normalization of solutions with respect to the Dirichlet data, in the proof of Theorem 4.3, may not be easily adapted in this case.

\subsection*{Acknowledgments} W. Feldman was partially supported by the NSF grants DMS-2009286 and DMS-2407235. I. Kim was partially supported by the NSF grant DMS-2153254. N. Pozar was partially supported by JSPS KAKENHI Kiban C Grant No. 23K03212.

\subsection{Notations and conventions} We list several notations and conventions which will be in force through the paper.

\begin{enumerate}[label = $\vartriangleright$]
\item We call a constant \emph{universal} if it only depends on $d$ and $\mu_+>0$, $\mu_- \in (0,1)$, a constant is called \emph{dimensional} if it only depends on the dimension.  
\item We will refer to universal constants by $C \geq1$ and $0<c \leq 1$ and allow such constants to change from line to line of the computation.  \item Denote $a \vee b = \max(a, b)$ and $a \wedge b = \min(a, b)$, respectively, for the maximum and minimum of two real numbers $a$ and $b$. 
\item We often abuse notation and write $\Omega(t)$ instead of $\Omega(u(t))$ etc.
   \item $u^*$ and $u_*$ denote the upper-semicontinuous envelope and the lower-semicontinuous envelope, respectively, of $u$ in both time and space. We write USC and LSC respectively as shorthand for upper-semicontinuous and lower-semicontinuous.
\item $F + H^1_0(U)$ refers to the space of functions in $H^1(U)$ with trace $F$ on $\partial U$.
\end{enumerate}

\section{Example of hysteresis}\label{s.ex-hysteresis} In order to understand the macroscopic effect of the hysteresis inherent in the slope pinning condition we consider a simple explicitly solvable radially symmetric example in dimension $d=2$. See \fref{hysteresis-loop} which will be further explained below.
\newcommand{\figScaleX}{.7}
\newcommand{\figScaleY}{.65}
\begin{figure}
\begin{tabular}{ll}
\begin{tikzpicture}[yscale=\figScaleY, xscale = \figScaleX]
\draw[->] (0,0) -- (3.5000,0.0000) node[right] {$|x|$};
\draw[->] (0,0) -- (0,5.5452+.5) node[above] {$u$};
\draw[black] plot coordinates {
    (0.0000,0.5000)
    (0.0263,0.4771)
    (0.0526,0.4548)
    (0.0790,0.4330)
    (0.1053,0.4118)
    (0.1316,0.3910)
    (0.1579,0.3707)
    (0.1842,0.3509)
    (0.2105,0.3316)
    (0.2369,0.3126)
    (0.2632,0.2940)
    (0.2895,0.2758)
    (0.3158,0.2580)
    (0.3421,0.2406)
    (0.3685,0.2235)
    (0.3948,0.2067)
    (0.4211,0.1902)
    (0.4474,0.1740)
    (0.4737,0.1581)
    (0.5000,0.1425)
    (0.5264,0.1272)
    (0.5527,0.1121)
    (0.5790,0.0973)
    (0.6053,0.0827)
    (0.6316,0.0684)
    (0.6580,0.0543)
    (0.6843,0.0404)
    (0.7106,0.0267)
    (0.7369,0.0133)
    (0.7632,0.0000)
};
\draw[black] plot coordinates {
    (0.0000,0.6667)
    (0.0263,0.6361)
    (0.0526,0.6064)
    (0.0790,0.5773)
    (0.1053,0.5490)
    (0.1316,0.5213)
    (0.1579,0.4943)
    (0.1842,0.4679)
    (0.2105,0.4421)
    (0.2369,0.4168)
    (0.2632,0.3920)
    (0.2895,0.3678)
    (0.3158,0.3440)
    (0.3421,0.3208)
    (0.3685,0.2979)
    (0.3948,0.2755)
    (0.4211,0.2536)
    (0.4474,0.2320)
    (0.4737,0.2108)
    (0.5000,0.1900)
    (0.5264,0.1696)
    (0.5527,0.1495)
    (0.5790,0.1297)
    (0.6053,0.1103)
    (0.6316,0.0912)
    (0.6580,0.0724)
    (0.6843,0.0539)
    (0.7106,0.0356)
    (0.7369,0.0177)
    (0.7632,0.0000)
};
\draw[black] plot coordinates {
    (0.0000,1.0000)
    (0.0263,0.9542)
    (0.0526,0.9096)
    (0.0790,0.8660)
    (0.1053,0.8235)
    (0.1316,0.7820)
    (0.1579,0.7415)
    (0.1842,0.7019)
    (0.2105,0.6631)
    (0.2369,0.6252)
    (0.2632,0.5881)
    (0.2895,0.5517)
    (0.3158,0.5161)
    (0.3421,0.4812)
    (0.3685,0.4469)
    (0.3948,0.4133)
    (0.4211,0.3804)
    (0.4474,0.3480)
    (0.4737,0.3162)
    (0.5000,0.2850)
    (0.5264,0.2544)
    (0.5527,0.2242)
    (0.5790,0.1946)
    (0.6053,0.1654)
    (0.6316,0.1368)
    (0.6580,0.1085)
    (0.6843,0.0808)
    (0.7106,0.0534)
    (0.7369,0.0265)
    (0.7632,0.0000)
};
\draw[black] plot coordinates {
    (0.0000,1.0000)
    (0.0263,0.9542)
    (0.0526,0.9096)
    (0.0790,0.8660)
    (0.1053,0.8235)
    (0.1316,0.7820)
    (0.1579,0.7415)
    (0.1842,0.7019)
    (0.2105,0.6631)
    (0.2369,0.6252)
    (0.2632,0.5881)
    (0.2895,0.5517)
    (0.3158,0.5161)
    (0.3421,0.4812)
    (0.3685,0.4469)
    (0.3948,0.4133)
    (0.4211,0.3804)
    (0.4474,0.3480)
    (0.4737,0.3162)
    (0.5000,0.2850)
    (0.5264,0.2544)
    (0.5527,0.2242)
    (0.5790,0.1946)
    (0.6053,0.1654)
    (0.6316,0.1368)
    (0.6580,0.1085)
    (0.6843,0.0808)
    (0.7106,0.0534)
    (0.7369,0.0265)
    (0.7632,0.0000)
};
\draw[black] plot coordinates {
    (0.0000,2.1363)
    (0.0489,2.0208)
    (0.0978,1.9105)
    (0.1468,1.8051)
    (0.1957,1.7040)
    (0.2446,1.6070)
    (0.2935,1.5138)
    (0.3424,1.4240)
    (0.3914,1.3374)
    (0.4403,1.2538)
    (0.4892,1.1731)
    (0.5381,1.0949)
    (0.5871,1.0191)
    (0.6360,0.9457)
    (0.6849,0.8744)
    (0.7338,0.8052)
    (0.7827,0.7379)
    (0.8317,0.6724)
    (0.8806,0.6087)
    (0.9295,0.5466)
    (0.9784,0.4860)
    (1.0273,0.4269)
    (1.0763,0.3693)
    (1.1252,0.3129)
    (1.1741,0.2579)
    (1.2230,0.2041)
    (1.2720,0.1514)
    (1.3209,0.0999)
    (1.3698,0.0494)
    (1.4187,0.0000)
};
\draw[black] plot coordinates {
    (0.0000,3.2726)
    (0.0686,3.0743)
    (0.1372,2.8884)
    (0.2058,2.7134)
    (0.2743,2.5480)
    (0.3429,2.3913)
    (0.4115,2.2425)
    (0.4801,2.1006)
    (0.5487,1.9653)
    (0.6173,1.8357)
    (0.6858,1.7116)
    (0.7544,1.5924)
    (0.8230,1.4778)
    (0.8916,1.3674)
    (0.9602,1.2610)
    (1.0288,1.1582)
    (1.0973,1.0588)
    (1.1659,0.9626)
    (1.2345,0.8695)
    (1.3031,0.7791)
    (1.3717,0.6914)
    (1.4403,0.6062)
    (1.5088,0.5233)
    (1.5774,0.4427)
    (1.6460,0.3642)
    (1.7146,0.2878)
    (1.7832,0.2132)
    (1.8518,0.1404)
    (1.9203,0.0694)
    (1.9889,0.0000)
};
\draw[black] plot coordinates {
    (0.0000,4.4089)
    (0.0866,4.1174)
    (0.1732,3.8482)
    (0.2597,3.5983)
    (0.3463,3.3649)
    (0.4329,3.1461)
    (0.5195,2.9401)
    (0.6060,2.7456)
    (0.6926,2.5613)
    (0.7792,2.3861)
    (0.8658,2.2193)
    (0.9524,2.0601)
    (1.0389,1.9078)
    (1.1255,1.7618)
    (1.2121,1.6216)
    (1.2987,1.4868)
    (1.3852,1.3570)
    (1.4718,1.2318)
    (1.5584,1.1110)
    (1.6450,0.9941)
    (1.7316,0.8811)
    (1.8181,0.7715)
    (1.9047,0.6653)
    (1.9913,0.5622)
    (2.0779,0.4620)
    (2.1644,0.3646)
    (2.2510,0.2698)
    (2.3376,0.1776)
    (2.4242,0.0877)
    (2.5107,0.0000)
};
\draw[black] plot coordinates {
    (0.0000,5.5452)
    (0.1034,5.1514)
    (0.2069,4.7930)
    (0.3103,4.4640)
    (0.4138,4.1601)
    (0.5172,3.8776)
    (0.6207,3.6138)
    (0.7241,3.3663)
    (0.8276,3.1332)
    (0.9310,2.9130)
    (1.0345,2.7042)
    (1.1379,2.5058)
    (1.2414,2.3168)
    (1.3448,2.1363)
    (1.4483,1.9636)
    (1.5517,1.7981)
    (1.6552,1.6391)
    (1.7586,1.4863)
    (1.8621,1.3390)
    (1.9655,1.1970)
    (2.0690,1.0598)
    (2.1724,0.9272)
    (2.2759,0.7989)
    (2.3793,0.6745)
    (2.4828,0.5539)
    (2.5862,0.4368)
    (2.6897,0.3230)
    (2.7931,0.2124)
    (2.8966,0.1048)
    (3.0000,0.0000)
};
\node[below] at (0.7632,0.0000) {$R_0$};
\draw (0.7632,-0.0500)--(0.7632,0.0500);
\node[below] at (3.0000,0.0000) {$R_1$};
\draw (3.0000,-0.0500)--(3.0000,0.0500);
\node[left] at (0.0000,0.5000) {$F_0$};
\draw (-0.0500,0.5000)--(0.0500,0.5000);
\node[left] at (0.0000,1.0000) {$\sigma F_0$};
\draw (-0.0500,1.0000)--(0.0500,1.0000);
\node[left] at (0.0000,5.5452) {$F_1$};
\draw (-0.0500,5.5452)--(0.0500,5.5452);
\end{tikzpicture}
&
\begin{tikzpicture}[yscale=\figScaleY, xscale = \figScaleX]
\draw[->] (0,0) -- (3.5000,0.0000) node[right] {$|x|$};
\draw[->] (0,0) -- (0,5.5452+.5) node[above] {$u$};
\draw[black!30!white] plot coordinates {
    (0.0000,5.5452)
    (0.1034,5.1514)
    (0.2069,4.7930)
    (0.3103,4.4640)
    (0.4138,4.1601)
    (0.5172,3.8776)
    (0.6207,3.6138)
    (0.7241,3.3663)
    (0.8276,3.1332)
    (0.9310,2.9130)
    (1.0345,2.7042)
    (1.1379,2.5058)
    (1.2414,2.3168)
    (1.3448,2.1363)
    (1.4483,1.9636)
    (1.5517,1.7981)
    (1.6552,1.6391)
    (1.7586,1.4863)
    (1.8621,1.3390)
    (1.9655,1.1970)
    (2.0690,1.0598)
    (2.1724,0.9272)
    (2.2759,0.7989)
    (2.3793,0.6745)
    (2.4828,0.5539)
    (2.5862,0.4368)
    (2.6897,0.3230)
    (2.7931,0.2124)
    (2.8966,0.1048)
    (3.0000,0.0000)
};
\draw[black!30!white] plot coordinates {
    (0.0000,3.6968)
    (0.1034,3.4343)
    (0.2069,3.1953)
    (0.3103,2.9760)
    (0.4138,2.7734)
    (0.5172,2.5851)
    (0.6207,2.4092)
    (0.7241,2.2442)
    (0.8276,2.0888)
    (0.9310,1.9420)
    (1.0345,1.8028)
    (1.1379,1.6706)
    (1.2414,1.5445)
    (1.3448,1.4242)
    (1.4483,1.3091)
    (1.5517,1.1987)
    (1.6552,1.0928)
    (1.7586,0.9908)
    (1.8621,0.8927)
    (1.9655,0.7980)
    (2.0690,0.7065)
    (2.1724,0.6181)
    (2.2759,0.5326)
    (2.3793,0.4497)
    (2.4828,0.3693)
    (2.5862,0.2912)
    (2.6897,0.2154)
    (2.7931,0.1416)
    (2.8966,0.0699)
    (3.0000,0.0000)
};
\draw[black!30!white] plot coordinates {
    (0.0000,2.7726)
    (0.1034,2.5757)
    (0.2069,2.3965)
    (0.3103,2.2320)
    (0.4138,2.0800)
    (0.5172,1.9388)
    (0.6207,1.8069)
    (0.7241,1.6831)
    (0.8276,1.5666)
    (0.9310,1.4565)
    (1.0345,1.3521)
    (1.1379,1.2529)
    (1.2414,1.1584)
    (1.3448,1.0682)
    (1.4483,0.9818)
    (1.5517,0.8991)
    (1.6552,0.8196)
    (1.7586,0.7431)
    (1.8621,0.6695)
    (1.9655,0.5985)
    (2.0690,0.5299)
    (2.1724,0.4636)
    (2.2759,0.3994)
    (2.3793,0.3372)
    (2.4828,0.2769)
    (2.5862,0.2184)
    (2.6897,0.1615)
    (2.7931,0.1062)
    (2.8966,0.0524)
    (3.0000,0.0000)
};
\draw[black!30!white] plot coordinates {
    (0.0000,2.7726)
    (0.1034,2.5757)
    (0.2069,2.3965)
    (0.3103,2.2320)
    (0.4138,2.0800)
    (0.5172,1.9388)
    (0.6207,1.8069)
    (0.7241,1.6831)
    (0.8276,1.5666)
    (0.9310,1.4565)
    (1.0345,1.3521)
    (1.1379,1.2529)
    (1.2414,1.1584)
    (1.3448,1.0682)
    (1.4483,0.9818)
    (1.5517,0.8991)
    (1.6552,0.8196)
    (1.7586,0.7431)
    (1.8621,0.6695)
    (1.9655,0.5985)
    (2.0690,0.5299)
    (2.1724,0.4636)
    (2.2759,0.3994)
    (2.3793,0.3372)
    (2.4828,0.2769)
    (2.5862,0.2184)
    (2.6897,0.1615)
    (2.7931,0.1062)
    (2.8966,0.0524)
    (3.0000,0.0000)
};
\draw[black!30!white] plot coordinates {
    (0.0000,2.2045)
    (0.0866,2.0587)
    (0.1732,1.9241)
    (0.2597,1.7991)
    (0.3463,1.6825)
    (0.4329,1.5731)
    (0.5195,1.4701)
    (0.6060,1.3728)
    (0.6926,1.2806)
    (0.7792,1.1931)
    (0.8658,1.1097)
    (0.9524,1.0300)
    (1.0389,0.9539)
    (1.1255,0.8809)
    (1.2121,0.8108)
    (1.2987,0.7434)
    (1.3852,0.6785)
    (1.4718,0.6159)
    (1.5584,0.5555)
    (1.6450,0.4971)
    (1.7316,0.4405)
    (1.8181,0.3858)
    (1.9047,0.3326)
    (1.9913,0.2811)
    (2.0779,0.2310)
    (2.1644,0.1823)
    (2.2510,0.1349)
    (2.3376,0.0888)
    (2.4242,0.0438)
    (2.5107,0.0000)
};
\draw[black!30!white] plot coordinates {
    (0.0000,1.6363)
    (0.0686,1.5372)
    (0.1372,1.4442)
    (0.2058,1.3567)
    (0.2743,1.2740)
    (0.3429,1.1957)
    (0.4115,1.1212)
    (0.4801,1.0503)
    (0.5487,0.9826)
    (0.6173,0.9179)
    (0.6858,0.8558)
    (0.7544,0.7962)
    (0.8230,0.7389)
    (0.8916,0.6837)
    (0.9602,0.6305)
    (1.0288,0.5791)
    (1.0973,0.5294)
    (1.1659,0.4813)
    (1.2345,0.4347)
    (1.3031,0.3895)
    (1.3717,0.3457)
    (1.4403,0.3031)
    (1.5088,0.2617)
    (1.5774,0.2214)
    (1.6460,0.1821)
    (1.7146,0.1439)
    (1.7832,0.1066)
    (1.8518,0.0702)
    (1.9203,0.0347)
    (1.9889,0.0000)
};
\draw[black!30!white] plot coordinates {
    (0.0000,1.0682)
    (0.0489,1.0104)
    (0.0978,0.9553)
    (0.1468,0.9025)
    (0.1957,0.8520)
    (0.2446,0.8035)
    (0.2935,0.7569)
    (0.3424,0.7120)
    (0.3914,0.6687)
    (0.4403,0.6269)
    (0.4892,0.5865)
    (0.5381,0.5474)
    (0.5871,0.5096)
    (0.6360,0.4729)
    (0.6849,0.4372)
    (0.7338,0.4026)
    (0.7827,0.3690)
    (0.8317,0.3362)
    (0.8806,0.3043)
    (0.9295,0.2733)
    (0.9784,0.2430)
    (1.0273,0.2135)
    (1.0763,0.1846)
    (1.1252,0.1565)
    (1.1741,0.1289)
    (1.2230,0.1020)
    (1.2720,0.0757)
    (1.3209,0.0499)
    (1.3698,0.0247)
    (1.4187,0.0000)
};
\draw[black!30!white] plot coordinates {
    (0.0000,0.5000)
    (0.0263,0.4771)
    (0.0526,0.4548)
    (0.0790,0.4330)
    (0.1053,0.4118)
    (0.1316,0.3910)
    (0.1579,0.3707)
    (0.1842,0.3509)
    (0.2105,0.3316)
    (0.2369,0.3126)
    (0.2632,0.2940)
    (0.2895,0.2758)
    (0.3158,0.2580)
    (0.3421,0.2406)
    (0.3685,0.2235)
    (0.3948,0.2067)
    (0.4211,0.1902)
    (0.4474,0.1740)
    (0.4737,0.1581)
    (0.5000,0.1425)
    (0.5264,0.1272)
    (0.5527,0.1121)
    (0.5790,0.0973)
    (0.6053,0.0827)
    (0.6316,0.0684)
    (0.6580,0.0543)
    (0.6843,0.0404)
    (0.7106,0.0267)
    (0.7369,0.0133)
    (0.7632,0.0000)
};
\node[below] at (0.7632,0.0000) {$R_0$};
\draw (0.7632,-0.0500)--(0.7632,0.0500);
\node[below] at (3.0000,0.0000) {$R_1$};
\draw (3.0000,-0.0500)--(3.0000,0.0500);
\node[left] at (0.0000,0.5000) {$F_0$};
\draw (-0.0500,0.5000)--(0.0500,0.5000);
\node[left] at (0.0000,5.5452) {$F_1$};
\draw (-0.0500,5.5452)--(0.0500,5.5452);
\node[left] at (0.0000,2.7726) {$\frac{1}{\sigma} F_1$};
\draw (-0.0500,2.7726)--(0.0500,2.7726);
\end{tikzpicture}
\\
\begin{tikzpicture}[yscale=\figScaleY, xscale = \figScaleX]
\draw[->] (0,0) -- (0,5.5452) node[right] {$F$};
\draw[->] (0,0) -- (4.5000,0) node[above] {$R$};
\draw[smooth,gray] plot coordinates {
    (1.0000,0.0000)
    (1.1118,0.2358)
    (1.2301,0.5095)
    (1.3520,0.8156)
    (1.4761,1.1496)
    (1.6015,1.5085)
    (1.7278,1.8896)
    (1.8545,2.2907)
    (1.9816,2.7103)
    (2.1088,3.1468)
    (2.2361,3.5990)
    (2.3635,4.0658)
    (2.4908,4.5463)
    (2.6181,5.0397)
    (2.7454,5.5452)
};
\draw[smooth,gray] plot coordinates {
    (1.0000,0.0000)
    (1.2143,0.2358)
    (1.4286,0.5095)
    (1.6429,0.8156)
    (1.8571,1.1496)
    (2.0714,1.5085)
    (2.2857,1.8896)
    (2.5000,2.2907)
    (2.7143,2.7103)
    (2.9286,3.1468)
    (3.1429,3.5990)
    (3.3571,4.0658)
    (3.5714,4.5463)
    (3.7857,5.0397)
    (4.0000,5.5452)
};
\draw[->] (1.6429,0.8156) -- (1.6429,1.2234) ;
\draw (1.6429,0.8156) -- (1.6429,1.6311) ;
\draw[smooth,black,->] plot coordinates {
    (1.6429,1.6311)
    (1.7278,1.8896)
    (1.8545,2.2907)
    (1.9816,2.7103)
};
\draw[smooth,black] plot coordinates {
    (1.9816,2.7103)
    (2.1088,3.1468)
    (2.2361,3.5990)
    (2.3635,4.0658)
    (2.4908,4.5463)
};
\draw (2.4908,3.4097) -- (2.4908,2.2732) ;
\draw[->] (2.4908,4.5463) -- (2.4908,3.4097) ;
\draw[smooth,black] plot coordinates {
    (1.6429,0.8156)
    (1.8571,1.1496)
    (2.0714,1.5085)
};
\draw[smooth,black,<-] plot coordinates {
    (2.0714,1.5085)
    (2.2857,1.8896)
    (2.4908,2.2732)
};
\filldraw (1.6429,0.8156) circle (0.0250);
\draw[dotted] (1.6429,0.8156) -- (0.0000,0.8156);
\node[left] at (0.0000,0.8156) {$F_0$};
\draw (-0.0500,0.8156)--(0.0500,0.8156);
\node[below] at (1.6429,0.0000) {$R_0$};
\draw (1.6429,-0.0500)--(1.6429,0.0500);
\filldraw (1.6429,1.6311) circle (0.0250);
\node[left] at (0.0000,1.6311) {$\sigma F_0$};
\draw[dotted] (1.6429,1.6311) -- (0.0000,1.6311);
\draw (-0.0500,1.6311)--(0.0500,1.6311);
\filldraw (2.4908,4.5463) circle (0.0250);
\node[left] at (0.0000,4.5463) {$F_1$};
\draw[dotted] (2.4908,4.5463) -- (0.0000,4.5463);
\draw (-0.0500,4.5463)--(0.0500,4.5463);
\filldraw (2.4908,2.2732) circle (0.0250);
\node[below] at (2.4908,0.0000) {$R_1$};
\draw (2.4908,-0.0500)--(2.4908,0.0500);
\node[left] at (0.0000,2.2732) {$\frac{1}{\sigma}F_1$};
\draw[dotted] (2.4908,2.2732) -- (0.0000,2.2732);
\draw (-0.0500,2.2732)--(0.0500,2.2732);
\node[right] at (2.7454,5.5452) {$\gamma_+$};
\node[right] at (4.0000,5.5452) {$\gamma_-$};
\end{tikzpicture}
&
\begin{tikzpicture}[yscale=\figScaleY, xscale = \figScaleX]

\draw[->] (0,0) -- (0,5.5452) node[right] {$F$};
\draw[->] (0,0) -- (4.5000,0) node[above] {$t$};

\draw (2,4.54) parabola (0,.8156);
\draw (2,4.54) parabola (4,.8156);
\node[below] at (4,0) {$T$};
\draw (4,-.05)--(4,.05);
\node[left] at (0.0000,0.8156) {$F_0$};
\draw (-0.0500,0.8156)--(0.0500,0.8156);
\node[left] at (0.0000,1.6311) {$\sigma F_0$};
\draw (-0.0500,1.6311)--(0.0500,1.6311);
\node[left] at (0.0000,4.5463) {$F_1$};
\draw (-0.0500,4.5463)--(0.0500,4.5463);

\node[left] at (0.0000,2.2732) {$\frac{1}{\sigma}F_1$};
\draw (-0.0500,2.2732)--(0.0500,2.2732);
\end{tikzpicture}
\end{tabular}
\caption{Hysteresis loop for radial solutions. Top left: profiles of advancing solution starting from $(R_0,F_0) \in \gamma_-$. Top right: profiles of receding solution starting at $(R_1,F_1) \in \gamma_+$. Bottom left: hysteresis loop diagram in $(R,F)$-plane. Bottom right: plot of $F(t)$. Parameters used to generate top and bottom set of pictures do not exactly match in order to better display the respective graphs.  The factor $\sigma$ is defined to be $\big(\frac{1+\mu_+}{1-\mu_-}\big)^{1/2}$.}
\label{f.hysteresis-loop}
\end{figure}
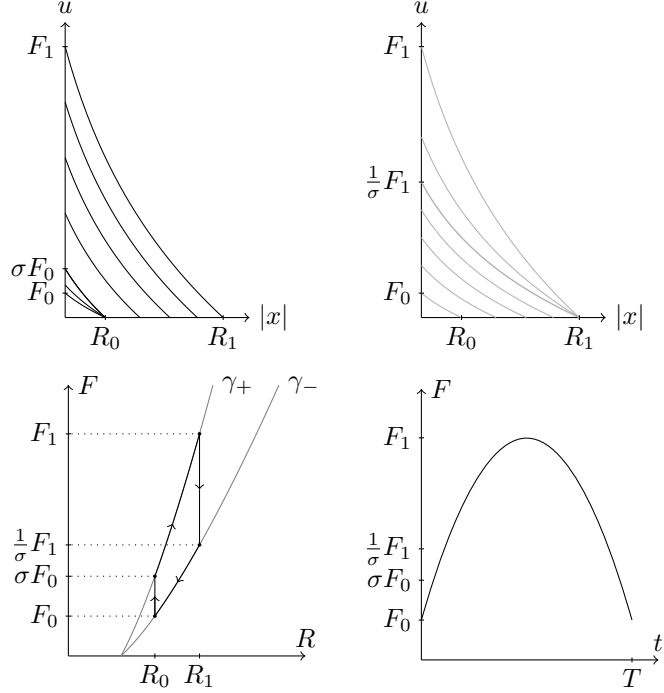

Denote the family of radially symmetric solutions $v_{\lambda,F}$ of
\[ 
\left\{\begin{aligned}
\Delta v_{\lambda,F} &= 0 \ \hbox{ in } \ \Omega(v_{\lambda,F}) \setminus \overline{B_1}, \\
v_{\lambda,F} &= F \ \hbox{ on } \ \partial B_1,\\
|\grad v_{\lambda,F}| &= \lambda \ \hbox{ on } \ \partial \Omega(v_{\lambda,F}) \setminus \overline{B_1}.
\end{aligned}\right.
\]
 These can be explicitly computed as
\[v_{\lambda,F}(x) = F\left(1-\frac{\log |x|}{\log \zeta(\lambda^{-1}F)}\right)_+\]
 where $\zeta : (0,\infty) \to (1,\infty)$ strictly monotone increasing is the inverse of $R \mapsto R \log R$ i.e.
 \[\zeta(s)\log \zeta(s) = s \ \hbox{ for } \ s > 0.\]
 Note that the radius of the support of $v_{\lambda,F}$ is $R = \zeta(\lambda^{-1}F)$.
 
 Now let us consider the radially symmetric solution of \eqref{e.stability-condition-intro}--\eqref{e.dynamic-slope-condition-intro} in $U = \R^2 \setminus B_1$ subject to the boundary condition $u(t) = F(t)$ on $\partial B_1$ with some forcing $F(t)$ and
\[ u_0 = v_{(1-\mu_-)^{1/2},F_0}.\]
The solution must be of the form
\[u(t) = v_{\lambda(t),F(t)} \ \hbox{ with } \ \lambda(t)^2 \in [1-\mu_-,1+\mu_+] \]
and the dynamic slope condition becomes
\[\pm\frac{d}{dt}R(t) = \pm\frac{d}{dt}\zeta(\lambda(t)^{-1}F(t)) >0 \ \hbox{ implies } \ \lambda(t) = 1\pm\mu_\pm.\]
In other words 
\[\pm\frac{d}{dt}R(t) > 0  \ \hbox{ implies } \ R(t) = \zeta((1\pm\mu_\pm)^{1/2}F(t)).\]
or viewing matters in the $(R,F)$ plane the state of the system is always in the region
\[\mathcal{S}  = \bigg\{(R,F) \in [1,\infty)\times (0,\infty) : \ \zeta((1+\mu_+)^{-1/2}F) \leq R \leq \zeta((1-\mu_-)^{-1/2}F)\bigg\}\]
and $R$ can only increase/decrease while on the respective boundary curves
\[\gamma_\pm = \{(R,F): \ R = \zeta((1\pm\mu_\pm)^{-1}F)\}.\]

\section{Basic space and time regularity properties of energy solutions} 
\label{s.basic-energy}

In this section we study energy solutions \EE{} establishing existence and various spatial and temporal regularity properties.

In \sref{energy-soln-spatial-reg} we recall several results from the literature on inward and outward minimizers of one-phase functionals.  These results, in particular, also apply to states satisfying the global stability condition of energy solutions \EE{}.

In \sref{energy-soln-temporal-reg} we show $BV$ time regularity of energy solutions using a typical Gr\"onwall argument via the energy dissipation inequality, see \lref{energy-soln-time-reg}.  With this regularity we can also establish that the energy dissipation balance holds with equality, see \lref{energy-diss-eq}.

Finally in \sref{energy-soln-lr-limits} we establish the main temporal regularity result of the section, \pref{lr-limits}. Abstract theory of bounded variation maps from $[0,T]$ into a metric space shows that $u(t)$ has left and right limits $u_\ell(t)$ and $u_r(t)$ at every time.  Then we use the monotonicity structure of the problem to show that the lower and upper-semicontinuous envelopes of an energy solution are exactly $u_\ell(t) \wedge u_r(t)$ and $u_\ell(t) \vee u_r(t)$ and these are energy solutions as well.  This is independently interesting, elucidating the structure of the jump discontinuities, and it also plays a key role in establishing the dynamic slope condition in \sref{energy-soln-motion-law}.

\subsection{inward and outward minimality and spatial regularity}\label{s.energy-soln-spatial-reg}
Let $\Lambda$ be an open region that contains  $\R^d\setminus U$.  In this section we will discuss the regularity of minimizers of the \emph{dissipation augmented energy}, defined above in \eref{E-def-intro},
\begin{equation*}\label{e.augmented-energy-def}
\mathcal{E}(\Lambda,u) := \mathcal{J}(u) + \textup{Diss}(\Lambda,\Omega(u)).
\end{equation*}
We will also abuse notation to write $\mathcal{E}(v,u)  = \mathcal{E}(\Omega(v),u)$ when $v$ is another non-negative function on $\R^n$.  

Note that if $u$ is an energy solution on $[0,T]$, from the global stability property at each time $u(t)$ is a minimizer for $\mathcal{E}(u(t),\cdot)$.

It is convenient to make a connection with the notions of inward and outward minimality for the Bernoulli functional $\mathcal{J}$.  First of all we introduce the notation for ${Q}>0$ and an open region $U \subset \R^d$ 
\[\mathcal{J}_{Q}(u) := \int_{U} |\grad u|^2 + {Q}{\bf 1}_{\{u>0\}} \ dx.\]
We have written and will continue to write $\mathcal{J} = \mathcal{J}_1$ abusing notation. 

Next we introduce the notions of inward and outward minimizers. These notions have appeared in the literature before, for example see the book \cite{VelichkovBook} for further references.
\begin{definition}
 $u \in H^1(U)$ is an {\it outward (resp. inward) minimizer} of $\mathcal{J}_{Q}(\cdot;U)$ if
\begin{enumerate}
\item The set $\Omega(u) = \{u>0\}$ is open and $u$ is harmonic in $\Omega(u)$.
\item For any $v \in u + H^1_0(U)$ with $v \geq u$ (resp. $v \leq u$)
\[\mathcal{J}_{Q}(u) \leq \mathcal{J}_{Q}(v).\]
\end{enumerate}
\end{definition}

Minimizers of $\mathcal{E}(\Lambda,\cdot)$ have a natural inward/outward $\mathcal{J}_{1\pm \mu_\pm}$ minimality property.

\begin{lemma}\label{l.inward-outward-min-props}
Let $\Lambda$ be an open set that contains $\R^d \setminus U$. Suppose that $u$ is a global minimizer of
\[ \mathcal{J}(v) + \textup{Diss}(\Lambda,\Omega(v)) \ \hbox{ over } \ v \in u + H^1_0(U).\]
Then $u$ is an outward minimizer for $\mathcal{J}_{1+\mu_+}$ and an inward minimizer for $\mathcal{J}_{1-\mu_-}$.  Furthermore if $B_r$ lies outside of $\overline{\Lambda}$ then $u$ minimizes $\mathcal{J}_{1+\mu_+}$ in $B_r$, and if $B_r \subset \Lambda$ then $u$ minimizes $\mathcal{J}_{1-\mu_-}$ in $B_r$.
\end{lemma}
\begin{proof}
Suppose $v \in u+H^1_0(U)$ with $v \geq u$.  Then $\{v>0\} \supset \{u>0\}$, so
\begin{equation}\label{e.outward-min-containments}
|\Lambda \setminus \Omega(u)| \geq |\Lambda \setminus \Omega(v)| \ \hbox{ and } \ |\Omega(u) \cap \Lambda| \leq |\Omega(v) \cap \Lambda|.
\end{equation}

We can apply \eref{outward-min-containments} along with the minimality property of $u$ to obtain
\begin{align*}
\mathcal{J}_{1+\mu_+}(u) &= \mathcal{J}_1(u) + \mu_+|\Omega(u)| \\
&=\mathcal{J}_1(u) + \mu_+|\Omega(u) \setminus \Lambda| + \mu_+|\Omega(u) \cap \Lambda|\\
&= \mathcal{J}_1(u) + \textup{Diss}(\Lambda,\Omega(u)) - \mu_-|\Lambda \setminus \Omega(u)|+ \mu_+|\Omega(u) \cap \Lambda|\\
&\leq \mathcal{J}_1(v) + \textup{Diss}(\Lambda,\Omega(v)) - \mu_-|\Lambda \setminus \Omega(v)|+ \mu_+|\Omega(v) \cap \Lambda|\\
&=\mathcal{J}_{1+\mu_+}(v).
\end{align*}
 Lastly note that if $v = u$ outside of  $B_r\subset \R^d \setminus \Lambda$ then \eref{outward-min-containments} holds with equalities.

 Symmetrical computations give the (inward) minimality property for $\mathcal{J}_{1-\mu_-}$.  
 \end{proof}
\subsubsection{Viscosity solution properties of inward / outward minimizers} The notions of inward and outward minimizers are reminiscent of viscosity sub and supersolutions.  Indeed there is a direct correspondence between the two notions.
\begin{lemma}\label{l.inward-outward-min-implies-vs}
Suppose that $S>0$ and $u$ is an inward minimizer of $\mathcal{J}_{S}$ in a domain $U$.  Then in the viscosity sense
\[|\grad u|^2 \geq {Q}  \ \hbox{ on } \partial \{u>0\} \cap U.\]
Similarly, if $u$ is an outward minimizer of $\mathcal{J}_{S}$ then in the viscosity sense
\[|\grad u|^2 \leq {Q}  \ \hbox{ on } \partial \{u>0\} \cap U.\]
\end{lemma}
In general the implication cannot go the other way: there are viscosity solutions, i.e. stationary points of the energy, which are not energy minimizers. 

It is well known that such inward and outward minimality properties imply the corresponding viscosity solution conditions, see for example \cite[Proposition 7.1]{VelichkovBook}. We present a proof anyway since we will use similar, but more involved, computations in the proof of \tref{ae-viscosity-prop}.

Before proceeding with the proof let us write down a corollary of \lref{inward-outward-min-props} and \lref{inward-outward-min-implies-vs} for energy solutions.
\begin{corollary}[Stability implies slope in pinning interval]\label{c.energy-stability-implies-viscosity-soln}
Suppose that $u$ is an energy solution on $[0,T]$, then  for each time $u(t)$ is a viscosity solution of
\[ 1-\mu_- \leq |\grad u(t)|^2 \leq 1+\mu_+ \quad \hbox{on } \partial \{u(t)>0\} \cap U.\]
\end{corollary}
Thus energy solutions satisfy the stability condition \eref{stability-condition-intro}.

\begin{proof} [Proof of \lref{inward-outward-min-implies-vs}]

{\bf (Subsolution)} Suppose that a smooth test function $\varphi$ touches $u$ from above in $\overline{\Omega(u)}$ strictly at $x_0 \in \partial \Omega(u)$ with $p:= \nabla\varphi(x_0) \neq 0$ and $\Delta\varphi (x_0) < 0$.  Let $\delta>0$ and consider the comparison function
\[ u_\delta(x)  := u(x) \wedge (\varphi(x) - \delta)_+.\]
Since the touching is strict $ \{u>0\} \setminus \{u_\delta>0\}$ is contained in a ball of radius $o_\delta(1)$ around $x_0$ as $\delta \to 0$.  In particular we can assume that $\delta$ is small enough so that $\varphi$ is superharmonic in $\{u_\delta < u\}$. By the inward minimality of $u$
\[\mathcal{J}_{Q}(u) \leq \mathcal{J}_{Q}(u_\delta).\]
Now note that $u_\delta$ extends to a superharmonic function in $\{u>0\}$ by
\[\bar{u}_\delta(x) :=\begin{cases} u_\delta(x) &x\in \Omega(u_\delta)\\
\varphi(x) - \delta &x \in \Omega(u) \setminus \Omega(u_\delta).
\end{cases} \]
So we can apply {\lref{intbyparts2}}, formula \eref{intbyparts2-supersoln}, and we find
\begin{align*} 
\mathcal{J}_{Q}(u) - \mathcal{J}_{Q}(u_\delta) 
& \geq \int_{\{u>0\} \setminus \{u_\delta>0\}}{Q}-|\nabla\varphi(x)|^2  \ dx\\
& \geq [({Q}-|p|^2)-o_{\delta}(1)]|\{{u}>0\} \setminus \{{u_\delta}>0\}|.
\end{align*}
Combining the previous
\[[({Q}-|p|^2)-o_{\delta}(1)]|\{{u}>0\} \setminus \{{u_\delta}>0\}| \leq 0\]
 so dividing through by $|\{{u}>0\} \setminus \{{u_\delta}>0\}|>0$ and taking $\delta \to 0$ we find
\[ {Q} \leq |p|^2.\]

{\bf(Supersolution)} Suppose that a smooth test function $\varphi$ touches $u$ from below strictly at $x_0 \in \partial \Omega(u)$ with $p:= \nabla \varphi(x_0) \neq 0$ and $\Delta\varphi (x_0) > 0$.  Let $\delta>0$ and consider the comparison function
\[ u_\delta(x)  = u(x) \vee (\varphi(x) + \delta).\]

Since the touching is strict $\{u_\delta>u\}$ is contained in a ball of radius $o_\delta(1)$ around $x_0$ as $\delta \to 0$.  In particular we can assume that $\delta$ is small enough so that $\Delta u_\delta \geq 0$. By the outward minimality property of $u$
\[
 \mathcal{J}_{S}(u) \leq \mathcal{J}_{S}(u_\delta).
\]
Applying {\lref{intbyparts2} (note $u_\delta$ not harmonic in $\{u_\delta > 0\}$ but in general subharmonic)}  we find
\begin{align*} 
\mathcal{J}_{S}(u) - \mathcal{J}_{S}(u_\delta) &\geq  \int_{\{u_\delta>0\} \setminus \{u>0\}}|\nabla u_\delta|^2 - {Q} \ dx \\
& =  \int_{\{u_\delta>0\} \setminus \{u>0\}}|\nabla\varphi(x)|^2 - {Q} \ dx\\
& \geq [(|p|^2-{Q})-o_{\delta}(1)]|\{u_\delta>0\} \setminus \{u>0\}|.
\end{align*}
Combining the previous 
\[ [(|p|^2-Q)-o_{\delta}(1)]|\{u_\delta>0\} \setminus \{u>0\}|\leq 0\]
so dividing through by $|\{u_\delta>0\} \setminus \{u>0\}|>0$ and taking $\delta \to 0$ we find
\[ |p|^2 \leq {Q} .\]
\end{proof}

\subsubsection{Regularity properties of inward / outward minimizers} Next we collect several basic regularity results of $\mathcal{J}_{Q}$ inward / outward minimizers.

\begin{lemma}\label{l.initial-fb-regularity}
\begin{enumerate}[label = (\roman*)]
\item (Lipschitz estimate) There is $C(d) \geq 1$ so that for any viscosity solution $u$  of
\[\Delta u = 0 \ \hbox{ in } \ \Omega(u) \cap B_2 \ \hbox{ and } \ |\grad u|^2 \leq {Q} \ \hbox{ on } \ \partial \Omega(u) \cap B_2\]
 we have
\[\|\grad u\|_{L^\infty(B_1)} \leq C(\sqrt{Q}+\|\grad u\|_{L^2(B_2)}).\]
In particular this holds for outward minimizers of $\mathcal{J}_{Q}$ in $B_2$ by \lref{inward-outward-min-implies-vs}.
\item (Non-degeneracy) There is $c(d,Q)>0$ so that if $u \in H^1(U)$ is an inward minimizer of $\mathcal{J}_{Q}$ in $U$, then
\[\sup_{B_r(x)} u \geq c r \quad\hbox{ for any } x\in \partial\{u>0\} \hbox{ and } B_r(x)\subset U.\]
\item\label{part.density-estimates} (Density estimates) For ${Q}_1\leq {Q}_2$, there is $c(d,{Q}_1,{Q}_2)>0$ so that if $u$ is an inward minimizer of $\mathcal{J}_{{Q}_1}$ and an outward minimizer of $\mathcal{J}_{{Q}_2}$ in $U$ then
\[c \leq \frac{|\Omega(u) \cap B_r(x)|}{|B_r|} \leq 1 - c \hbox{ for any } x\in \partial\Omega(u) \hbox{ with } B_r(x) \subset U. \]
\item\label{part.perimeter-esimate} (Perimeter estimate) There is $C(d)>0$ such that if $u$ is an inward minimizer of $\mathcal J_S$ in $B_2$ then 
\[\textup{Per}(\Omega(u);B_1) \leq C\sqrt{Q}(1+\|\grad u\|_{L^2(B_2)}).\]
\item (Hausdorff dimension) If $u$ satisfies the conclusions of \partref{density-estimates} and perimeter estimate \partref{perimeter-esimate} in $B_2$ then
\[\mathcal{H}^{d-1}(\partial \Omega(u) \cap B_{1}) \leq C(1+\|\grad u\|_{L^2(B_2)})\]
where $C$ depends on the constants in \partref{density-estimates} and \partref{perimeter-esimate}.
\end{enumerate}
\end{lemma}

\begin{proof}
See the books \cite{VelichkovBook,CaffarelliSalsa} for presentations of the proofs and citations for their original appearances in the literature, Lipschitz continuity of viscosity solutions is in \cite[Lemma 11.19]{CaffarelliSalsa}, non-degeneracy is in \cite[Lemma 4.4]{VelichkovBook}, density estimates are in \cite[Lemma 5.1]{VelichkovBook}, perimeter estimates are in \cite[Lemma 5.6]{VelichkovBook}, Hausdorff dimension estimates in \cite[Lemma 5.9]{VelichkovBook}.
\end{proof}

\begin{remark}
The $C^{1,\beta}$ regularity of $\mathcal{E}(u,\cdot)$ minimizers is known due to the very recent result of Ferreri and Velichkov \cite{FerreriVelichkov} when $\{u>0\}$ is $C^{1,\alpha}$.
\end{remark}

\subsection{Bounded variation regularity in $t$}\label{s.energy-soln-temporal-reg}

BV regularity in time is typical for solutions of (global) dissipative evolution problems \cite{mielke2015book}.  Following the standard argument, we establish the $\textup{BV}$ in time regularity using a Gr\"onwall type argument with the energy dissipation inequality.

\begin{lemma}\label{l.energy-soln-time-reg}
Suppose $u$ is an energy solution \EE{} on $[0,T]$. Then
\[{\bf 1}_{\Omega(u(t))} \in \textup{BV}([0,T]; L^1(\R^d))\]
and
\[\mathcal{J}(u(t)), \ P(u(t))  \in \textup{BV}([0,T];\R),\]
where recall the pressure is $P(u(t))   = \int_{\partial \Omega(t)} |\nabla u| dS= \int_{\partial U } \frac{\partial u}{\partial n} dS = F(t)^{-1}\int |\nabla u(t)|^2 \ dx$.
\end{lemma}

Before beginning the proofs we make a remark about the definition of energy solution.
\begin{remark}\label{r.energy-dissipation-equality}
Energy solutions will also satisfy an \emph{energy dissipation balance} condition.  Define the total variation of the dissipation distance along a path $\Lambda(t)$ of finite measure subsets of $\R^d$
\begin{equation}\label{e.dissbar-def}
\Dissbar(\Lambda(\cdot);[s,t]) := \sup \left\{ \sum_{k=1}^N \textup{Diss}(\Lambda({t_{k-1}}),\Lambda({t_{k}})): (t_k)_{k=0}^N \hbox{ partitions $[s,t]$}\right\}.
\end{equation}
Then let $u(t)$ be an energy solution on $[0,T]$, i.e. satisfying the forcing, stability, and energy dissipation inequality conditions of \EE{}. One can show immediately, by applying \eref{diss-inequality} on each subinterval of an arbitrary partition of $[t_0,t_1]$, that 
\begin{equation}\label{e.dissbar-inequality}
 \mathcal{J}(u({t_0}))-\mathcal{J}(u({t_1}))  + \int_{t_0}^{t_1}2 \dot{F}(t)P(t) \ dt \geq \Dissbar(u(\cdot);[t_0,t_1]).
\end{equation}
With more work (see \lref{energy-diss-eq} later) combining with the stability property one can also show the identity
\begin{equation}\label{e.energy-diss-eq}
\mathcal{J}(u({t_0}))-\mathcal{J}(u({t_1}))  + \int_{t_0}^{t_1}2 \dot{F}(t)P(t) \ dt = \Dissbar(u(\cdot);[t_0,t_1]).
\end{equation}
In some works, for example \cite{alberti2011}, the energy dissipation balance \eref{energy-diss-eq} is used in place of the energy dissipation inequality \eref{diss-inequality} as part of the definition of energy solution.  This difference of definition is just a matter of preference at least in this problem.
\end{remark}

\begin{remark}
It is quite natural, and important for applications, to consider general $F(t,x)$ and allow for $F$ to be only $\textup{BV}$ regular in time.  However, this generality does add serious complications which would significantly lengthen the presentation, and it is not so relevant to the goals of the present work.  
\end{remark}

\begin{proof}[Proof of \lref{energy-soln-time-reg}]
By the energy dissipation inequality \eref{dissbar-inequality}
\begin{equation}\label{above_0}
 \mathcal{J}(u(0))-\mathcal{J}(u(t))  + \int_{0}^{t}2 \dot{F}(s)P(u(s)) \ ds \geq  \overline{\textup{Diss}}(\Omega(u(\cdot));[0,t]) 
\geq \mu_+ \wedge \mu_- [{\bf 1}_{\Omega(\cdot)}]_{\textup{BV}([0,t];L^1)}
 \end{equation}
In the remainder of the proof we will denote the Dirichlet energy
\[\mathcal{D}(t) := \int_{\Omega(u(t))} |\nabla u(t,x)|^2 \ dx = F(t)P(u(t)).\]
Note that $\mathcal{D}(t) \leq \mathcal{J}(u(t))$ so in particular
\begin{align*}
\mathcal{D}(t) &\leq \mathcal{J}(u(0))+\int_{0}^{t}2 \dot{F}(s)P(u(s)) \ ds \\
&= \mathcal{J}(u(0))+\int_{0}^{t}2 \frac{d}{ds}(\log F)(s) \mathcal{D}(s) \ ds
\end{align*}
for all $0 \leq t \leq T$ so by Gr\"onwall
\[\mathcal{D}(t) \leq \mathcal{J}(u(0)) \frac{F(t)^2}{F(0)^2}.\]
Thus $P(u(t)) \leq F(0)^{-2} F(t)$, and from \eqref{above_0}  we have 
\begin{align*}
\mathcal{J}(u(t)) + \mu_+ \wedge \mu_- [{\bf 1}_{\Omega(\cdot)}]_{\textup{BV}([0,t];L^1)} &\leq \mathcal{J}(u(0))\left(1+\frac{1}{F(0)^2}\int_0^t2\dot{F}(s)F(s) \ ds\right) \\
&= \mathcal{J}(u(0))\left(1+\frac{F(t)^2-F(0)^2}{F(0)^2}\right)\\
&=\mathcal{J}(u(0))\frac{F(t)^2}{F(0)^2}.
\end{align*}
To summarize,
\begin{equation}\label{e.energy-ineq-BV}
\mathcal{J}(u(t)) + \mu_+ \wedge \mu_- [{\bf 1}_{\Omega(\cdot)}]_{\textup{BV}([0,t];L^1)}\leq \mathcal{J}(u(0))\frac{F(t)^2}{F(0)^2}.
\end{equation}
In particular
\begin{equation}\label{bv_1}
[{\bf 1}_{\Omega(\cdot)}]_{\textup{BV}([0,t];L^1)} \leq \frac{1}{\mu_+\wedge \mu_-} \mathcal{J}(u(0))\frac{F(t)^2}{F(0)^2}.
\end{equation}
This gives the BV in time estimate of ${\bf 1}_{\Omega(t)}$.

Now we turn to the BV estimate of the Dirichlet energy.  Focusing on the other term on the left hand side in \eref{energy-ineq-BV}, and using the fact that $\mathcal{J}(u(t)) = \mathcal{D}(t) + |\Omega(u(t))|$,
\[\mathcal{D}(t) \leq \frac{F(t)^2}{F(0)^2}\mathcal{D}(0)+\frac{F(t)^2}{F(0)^2}|\Omega(u(0))| - |\Omega(u(t))|\]
or rearranging this
\begin{align}\mathcal{D}(t) - \mathcal{D}(0) &\leq \frac{\mathcal{D}(0)}{F(0)^{2}}(F(t)^2 - F(0)^2)+\frac{1}{F(0)^{2}}(F(t)^2-F(0)^2)|\Omega(u(0))|\notag\\
&\quad  +|\Omega(u(t)) \Delta \Omega(u(0))|\notag\\
& = \frac{\mathcal{J}(u(0))}{F(0)^{2}}(F(t)^2 - F(0)^2)+|\Omega(u(t)) \Delta \Omega(u(0))|.\label{e.energy-ineq-D-upper}
\end{align}
Global stability implies a bound in the other direction
\begin{align*}
 \mathcal{J}(u(0)) &\leq J\left(\frac{F(0)}{F(t)}u(t)\right) + \mu_+ \vee \mu_- |\Omega(u(0)) \Delta \Omega(u(t))|,
 \end{align*}
 which expands out as
 \[\mathcal{D}(0)+|\Omega(u(0))| \leq \frac{F(0)^2}{F(t)^2}\mathcal{D}(t)+|\Omega(u(t))| + \mu_+ \vee \mu_- |\Omega(u(0)) \Delta \Omega(u(t))|,\]
 which we can use to find
\begin{equation}\label{e.energy-ineq-D-lower}
\mathcal{D}(0) - \mathcal{D}(t)  \leq \frac{\mathcal{D}(t)}{F(t)^{2}}(F(0)^2 - F(t)^2)+(1+ \mu_+ \vee \mu_-) |\Omega(u(0)) \Delta \Omega(u(t))|.
 \end{equation}
 
 Since $0$ and $t$ can be replaced by arbitrary $t_0 < t_1 \in [0,T]$, for any partition of $[0,T]$ we can apply \eref{energy-ineq-D-lower} and \eref{energy-ineq-D-upper} together to find:
\begin{align*}
\sum_j |\mathcal{D}(t_{j+1}) - \mathcal{D}(t_j)| &\leq \sum_j\bigg[\max\left\{\frac{\mathcal{J}(u(t_j))}{F(t_j)^2},\frac{\mathcal{J}(u(t_{j+1}))}{F(t_{j+1})^2}\right\}\left|F(t_{j+1})^2 - F(t_{j})^2\right|\\
&\quad \quad  + (1+\mu_+ \vee \mu_-) |\Omega(u(t_{j+1})) \Delta \Omega(u(t_j))|.
\end{align*}
Applying  \eref{energy-ineq-BV} again for $F(t)^{-2}\mathcal{J}(u(t)) \leq F(0)^{-2}\mathcal{J}(u(0))$, we obtain that 
$$
\sum_j |\mathcal{D}(t_{j+1}) - \mathcal{D}(t_j)| \leq \frac{\mathcal{J}(u(0))}{F(0)^2}[F^2]_{\textup{BV}([0,T])}+(1+\mu_+ \vee \mu_-)[{\bf 1}_{\Omega(\cdot)}]_{\textup{BV}([0,T];L^1)}.
$$

Thus we conclude that $\mathcal{D}(t) \in BV([0,T])$ with the estimate, applying \eqref{bv_1},
\[[\mathcal{D}]_{BV([0,T])} \leq \left[[F^2]_{\textup{BV}([0,T])}\\
+\frac{1+\mu_+ \vee \mu_-}{\mu_+\wedge \mu_-} F(t)^2\right]\frac{\mathcal{J}(u(0))}{F(0)^2}.\] Finally since $\mathcal{J}(u(t)) = \mathcal{D}(t) + |\Omega(u(t))|$ and both terms on the right are in $BV([0,T])$ then so is $\mathcal{J}(u(t))$.  Similarly $P(u(t)) = F(t)^{-1}\mathcal{D}(t)$ is in $BV([0,T])$ since $F(t)$ is bounded from below.
\end{proof}

From these regularity properties we can establish that the stability and energy dissipation inequality properties of energy solutions force the \emph{energy dissipation equality}. 
\begin{lemma}[Energy dissipation equality]\label{l.energy-diss-eq}
If $u$ is an energy solution on $[0,T]$ then for all $s\leq t$ in $[0,T]$
\[ \mathcal{J}(u(s))-\mathcal{J}(u(t))  + \int_{s}^{t} 2 \dot{F}(\tau)P(u(\tau)) \ d\tau = \overline{\textup{Diss}}(\Omega(u(\cdot));[s,t]).\]
\end{lemma}
\begin{proof}
We need to establish the upper bound to complement \eref{diss-inequality}. This will follow from the time regularity and the global stability property.  We may take $s = 0$ and $t = T$.
Let $\ep>0$. By the continuity of $\log F(t)$, we can choose a (finite) partition $0 = t_0< t_1 < \cdots < t_N = T$ of $[0,T]$ such that 
 \begin{equation}\label{e.partition-choice-cond}
\sup_j(t_{j+1}-t_j) \leq \ep \ \hbox{ and } \ \sup_{t \in [t_j,t_{j+1}]} \frac{F(t)}{F(t_{j+1})} \leq 1 + \ep.
 \end{equation}
 We write the energy dissipation using a telescoping sum
\begin{equation}\label{e.energy-telescope}
  \mathcal{J}(u(0))-\mathcal{J}(u(T)) = \sum_{j=0}^{N-1} [\mathcal{J}(u(t_{j}))-\mathcal{J}(u(t_{j+1}))].
  \end{equation}
  Applying the global stability condition for $u(t_j)$, 
  \begin{align*}
  \mathcal{J}(u(t_{j}))-\mathcal{J}(u(t_{j+1})) &\leq \mathcal{J}\left(\tfrac{F(t_{j})}{F(t_{j+1})}u(t_{j+1})\right)+\textup{Diss}(\Omega(t_j),\Omega(t_{j+1}))-\mathcal{J}(u(t_{j+1})) \\
  &= \textup{Diss}(\Omega(t_j),\Omega(t_{j+1}))+  \underbrace{(\tfrac{F(t_j)^2}{F(t_{j+1})^2} - 1)\mathcal{D}(u(t_{j+1}))}_{\text{:=A}}
  \end{align*}
Now from the fact that $\mathcal{D}(u(t)) = F(t) P(u(t))$, 
\[
A= \frac{F(t_j)^2 - F(t_{j+1})^2}{F(t_{j+1})}P(u(t_{j+1}))=\int_{t_j}^{t_{j+1}} 2 \frac{F(t)}{F(t_{j+1})} \dot{F}(t) P(u(t_{j+1})) \ dt.
\]
From \eref{partition-choice-cond} we have
\begin{align*}
A &=\int_{t_j}^{t_{j+1}} 2 \frac{F(t)}{F(t_{j+1})} \dot{F}(t) P(u(t_{j+1})) \ dt \notag\\
&\quad \quad \leq \int_{t_j}^{t_{j+1}} 2 \frac{F(t)}{F(t_{j+1})} \dot{F}(t) P(u(t)) \ dt + 2(1+\ep)\|\dot{F}\|_\infty [P(u(\cdot))]_{\textup{BV}((t_j,t_{j+1}])}(t_{j+1}-t_j)\notag\\
&\quad \quad \leq (1+\ep)\int_{t_j}^{t_{j+1}} 2 \dot{F}(t) P(u(t)) \ dt + 2(1+\ep)\ep\|\dot{F}\|_\infty [P(u(\cdot))]_{\textup{BV}((t_j,t_{j+1}])}
\end{align*}
Combining the previous estimates into \eref{energy-telescope} we find
\begin{align*}
\mathcal{J}(u(0))-\mathcal{J}(u(T)) &\leq(1+\ep)\int_{0}^{T} 2 \dot{F}(t) P(u(t)) \ dt +\sum_{j=0}^{N-1}\textup{Diss}(\Omega(t_{j}),\Omega(t_{j+1})) \\
&\quad \quad + \sum_{j=0}^{N-1}2(1+\ep)\ep\|\dot{F}\|_\infty [P(u(\cdot))]_{\textup{BV}((t_j,t_{j+1}])} \\
&\leq \overline{\textup{Diss}}(\Omega(u(\cdot));(0,T))+(1+\ep)\int_{0}^{T} 2 \dot{F}(t) P(u(t)) \ dt      \\
&\quad \quad + 2(1+\ep)\ep \|\dot{F}\|_\infty [P(u(\cdot))]_{\textup{BV}((0,T])}.
\end{align*}
  Since $\ep>0$ was arbitrary we get the result.
  \end{proof}

\subsection{Left and right limits and monotonicity of jumps}\label{s.energy-soln-lr-limits}
The $\textup{BV}$ time regularity of ${\bf 1}_{\Omega(u(t))}$ allows us to establish further temporal left and right continuity properties of the evolution at every time.  Recall the dissipation augmented energy functional \eref{E-def-intro},
\begin{equation*}
\mathcal{E}(\Lambda,u) := \mathcal{J}(u) + \textup{Diss}(\Lambda,\Omega(u)).
\end{equation*}
First we discuss a monotonicity property for minimizers of $\mathcal{E}$.

\subsubsection{No crossing property of $\mathcal{E}$ minimizers}

The next lemma shows that minimizers of the dissipation distance augmented energy $\mathcal{E}$, defined in \eref{E-def-intro}, satisfy a certain ordering property.  

Suppose that $u_1$ and $u_2$ are arbitrary $H^1(U)$ functions. Note that
 \[ \mathcal{J}(u_1 \vee u_2) + \mathcal{J}(u_1\wedge u_2) =  \mathcal{J}(u_1) + \mathcal{J}(u_2)\]
 and
 \[ \textup{Diss}(\Lambda,\Omega(u_1 \vee u_2)) + \textup{Diss}(\Lambda,\Omega(u_1 \wedge u_2)) = \textup{Diss}(\Lambda,\Omega(u_1))+ \textup{Diss}(\Lambda,\Omega(u_2)).\]
  Combining the previous two equations yields
  \begin{equation}\label{e.E-additivity-property}
   \mathcal{E}(\Lambda,u_1 \vee u_2) +  \mathcal{E}(\Lambda,u_1 \wedge u_2) = \mathcal{E}(\Lambda,u_1 ) +  \mathcal{E}(\Lambda, u_2).
   \end{equation}
  It is standard to derive from this additivity property that the pointwise minimum and maximum of two minimizers with the same Dirichlet condition are energy minimizers as well.  Adding in the strong maximum principle for harmonic functions we can show an ordering property of the collection of energy minimizers, again this is a sufficiently standard idea that we do not have a particular original reference.

\begin{lemma}[No crossing]\label{l.no-crossing}
Let $\Lambda$ be an open set containing $\R^d \setminus U$.  
If $u_1$ and $u_2$ both minimize $\mathcal{E}(\Lambda,\cdot)$ with respect to $H^1_0(U)$ perturbations and $u_1 - u_2 \in H^1_0(U)$. Then $u_1 \wedge u_2$ and $u_1 \vee u_2$ are also minimizers and $u_1$ and $u_2$ are ordered in each connected component of $\Omega(u_1) \cup \Omega(u_2) = \Omega(u_1 \vee u_2)$.

\end{lemma}
\begin{remark}\label{r.multiple-components}
To be clear the ordering between $u_1$ and $u_2$ may differ in the different connected components of $\Omega(u_1) \cup \Omega(u_2)$.   For example, given two ordered minimizers the same scenario can be repeated far away with the reverse ordering to create unordered minimizers.
\end{remark}

\begin{proof}[Proof of \lref{no-crossing}]Recalling \eref{E-additivity-property} we have
  \[ \mathcal{E}(\Lambda,u_1 \vee u_2) +  \mathcal{E}(\Lambda,u_1 \wedge u_2) = \mathcal{E}(\Lambda,u_1 ) +  \mathcal{E}(\Lambda, u_2).\]
 On the other hand, since $u_1 = u_2$ on $\partial U$ in the trace sense, also $u_1 \wedge u_2, u_1 \vee u_2  = u_1 = u_2$ on $\partial U$ in the trace sense.  So, from the minimizer assumption on $u_1$ and $u_2$, 
  we conclude that all four terms in the above equality must be equal.  Thus $u_1$, $u_2$, $u_1 \wedge u_2$, and $u_1 \vee u_2$ are all minimizers of $\mathcal{E}(\Lambda,\cdot)$ over the admissible class.  In particular each one is harmonic in its positivity set.  By unique continuation for harmonic functions this means that $u_1$ and $u_2$ are ordered in each connected component of $U \cap \{u_1 \vee u_2 > 0\}$.  Note that by harmonicity again if $u_1 \neq u_2$ in a given connected component of $U \cap \{ u_1 \vee  u_2 >0\}$ then the ordering is strict in it. 
\end{proof}

\subsubsection{Abstract left-right limits} Next we discuss some general facts about bounded variation maps on intervals of the line $\R$.  Suppose $(Y, d_Y)$ is a complete metric space.
\begin{definition} 
Say $f : [a,b] \to Y$ is a bounded variation map if
\[[f]_{\textup{BV}([a,b];Y)} := \sup_{a =t_0 \leq t_1 \leq\cdots \leq t_{N+1} = b} \sum_{j=0}^N d(f(t_{j+1}),f(t_j)) < +\infty.\]
On open or half open intervals define
\[[f]_{\textup{BV}([a,b);Y)}:= \lim_{b' \nearrow b}[f]_{\textup{BV}([a,b'];Y)},\]
or similar for $(a,b]$ and $(a,b)$.  The limit exists by monotonicity.
\end{definition}
\begin{definition}
Let $f: [a,b] \to Y$.  For each $t \in [a,b]$, if the limits exist define the left limit $f_\ell(t)$ at $t \in (a,b]$
\[f_\ell(t) := \lim_{s \to t -} f(s)\]
and the right limit $f_r(t)$ at $t \in [a,b)$
\[ f_r(t):= \lim_{s \to t +} f(s).\]
\end{definition}

\begin{lemma}\label{l.BV-abstract-lr}
Let $(Y,d)$ be a complete metric space and $f$ be a bounded variation map $[a,b] \to Y$. Then $f$ has left (resp. right) limit at each $t \in (a,b]$ (resp. $[a,b)$) and $f$ has at most countably many jump discontinuities.
\end{lemma}

The proof is standard, but we give a sketch for convenience. 

\begin{proof}

For the directional limits note that for any sequence $s_j \nearrow t \in (a,b]$
\[\sum_j d_Y(f(s_j),f(s_{j+1})) \leq [f]_{\textup{BV}([a,b];Y)} < + \infty\]
and so $f(s_j)$ is Cauchy in $Y$.  The limits must agree on different approaching sequences by a typical interlacement argument with pairs of sequences. A symmetric argument produces the right limits.  

Now, first arguing for any finite set $\mathcal{T}$ of times, and then taking the supremum over all finite sets
\[\sum_{t \in [a,b]}d_Y(f_\ell(t),f_r(t)) \leq [f]_{\textup{BV}([a,b];Y)}.\]
In particular $f_\ell(t) = f_r(t)$ except for at most countably many times.

\end{proof}
 \subsubsection{Characterization of the upper and lower semicontinuous envelopes for energy solutions.} \lref{BV-abstract-lr} and \lref{energy-soln-time-reg} together yield the existence of left and right limits of $\Omega(u(t))$, $\mathcal{J}(u(t))$, and $P(u(t))$ at all times.  Our next result says  that  $u(t)$ also has left and right limits in the uniform metric, even though we do not necessarily establish that $t \to u(t)$ is a $\textup{BV}$ in time map into $C_c(\R^d)$. The left and right limits of all the previous quantities are consistent, i.e. $(P\circ u)_\ell(t) = P(u_\ell(t))$ etc. Furthermore the left and right limits of $u(t)$ satisfy certain global minimality properties for the energy plus dissipation distance.

The special structure of the problem comes in when we prove certain monotonicity properties of all the jumps, namely properties \partref{lr-limits-jumps}-\partref{lr-limits-redefine} below.  This allows us to make a simple classification of the upper and lower semicontinuous envelopes of energy solutions, showing that they are energy solutions themselves.  This simple jump structure will be very useful in \sref{energy-soln-motion-law}, where we prove a certain weak version of the dynamic slope condition. This role makes the following proposition a central result of this section.

Semicontinuous envelopes are typically important in studying the geometric properties of interface evolution problems.  The fact that these envelopes are themselves energy solutions is extremely useful to us in \sref{energy-soln-motion-law}. 

In service of analyzing the properties of energy solutions at jump times and simplifying repetitive notations let us introduce the set
\begin{equation}\label{e.replacement-U-values}
\mathcal{U}(t) = \{u_\ell(t),u(t),u_r(t),u_{\ell}(t) \wedge u_r(t), u_\ell(t) \vee u_r(t)\}
\end{equation}
which is a kind of multi-valued version of $u(t)$ indexing the values taken by $u(t)$ ``near" time $t$ as well as the upper and lower envelopes.  The set $\mathcal{U}(t)$ comes with a natural partial order induced by the time variable.  Since $u_\ell(t)$ is the limit from the left of $u(t)$ it is ``before" all the other elements of $\mathcal{U}(t)$, and since $u_r(t)$ is the limit from the right it is ``after" all the other elements of $u(t)$.  Specifically we define the partial order
\begin{equation}\label{e.time-ordering-def}
u_\ell(t) \ \trianglelefteq \ u(t), \  u_{\ell}(t) \wedge u_r(t), \  u_\ell(t) \vee u_r(t) \ \trianglelefteq \  u_r(t).
\end{equation}
To be clear this partial ordering is purely related to the ``time" of the elements of $\mathcal{U}(t)$ and is not related to the spatial ordering between the different elements of $\mathcal{U}(t)$.

The core step in the following proposition is part \partref{lr-limits-minimizer}. Essentially we show that every element of $\mathcal{U}(t)$ is a valid intermediate state for the evolution at time $t$.  More specifically, the energy dissipation relation is still satisfied as long as the system jumps from $u_\ell(t)$ to any element of $\mathcal{U}(t)$ and then to $u_r(t)$.  

\begin{proposition}\label{p.lr-limits}
Suppose $u$ is an energy solution on $[0,T]$ then:
\begin{enumerate}[label = (\roman*)]
\item\label{part.lr-limits} (Left and right limits exist and are consistent) At every $t \in [0,T]$, $u(t)$, $\mathcal{J}(u(t))$, $P(u(t)$, and $\Omega(u(t))$ have left and right limits and 
\begin{equation}\label{e.consistency-lr-limits}
\begin{array}{c}
\lim_{s \to_{\ell / r} t} \Omega(u(t)) = \Omega(u_{\ell / r}(t)), \ \lim_{s \to_{\ell / r} t} \mathcal{J}(u(t)) = \mathcal{J}(u_{\ell / r}(t)), \\
 \hbox{ and } \  \lim_{s \to_{\ell / r} t} P(u(t)) = P(u_{\ell / r}(t))
 \end{array}
\end{equation}
\item\label{part.lr-limits-minimizer} (Jumps minimize energy plus dissipation) At every $t \in [0,T]$ for any $v,w \in \mathcal{U}(t)$ (from \eref{replacement-U-values}), with $v \trianglelefteq w$ (from \eref{time-ordering-def})
\[w \ \hbox{ minimizes } \ \mathcal{E}(v,\cdot) \ \hbox{ over } \ F(t) + H^1_0(U).\]
Note that in particular this implies all $v \in \mathcal{U}(t)$ are globally stable since $v \trianglelefteq v$.
\item\label{part.lr-limits-jumps}  (Jumps are monotone per component) At every $t \in [0,T]$
\[u_\ell(t) \wedge u_r(t) \leq u(t) \leq u_\ell(t) \vee u_r(t).\]
In particular at each $t \in [0,T]$ the upper semicontinuous and lower semi-continuous envelopes are $u^*(t) = u_r(t) \vee u_\ell(t)$ and $u_*(t) = u_r(t) \wedge u_\ell(t)$.
\item\label{part.lr-limits-redefine}  (Re-definition at jumps) Suppose that $v:[0,T] \times \R^d \to \R$ satisfies the property: 
\[\hbox{for each $t \in [0,T]$ } \ v(t) \in \mathcal{U}(t) \ \hbox{ (from \eref{replacement-U-values})}.\] Then $v(t)$ is also an energy solution on $[0,T]$.  In particular the upper and lower semicontinuous envelopes $u^*(t)$ and $u_*(t)$ are also energy solutions.
\end{enumerate}
\end{proposition}
Besides the independent interest in understanding the time regularity properties of energetic solutions, this result plays a key role in establishing the validity of the free boundary motion law in \tref{ae-viscosity-prop}.
\begin{remark}
Part \partref{lr-limits-jumps} does allow the possibility that the free boundaries $\partial \Omega(u_\ell(t))$ and $\partial \Omega(u_r(t))$ are distinct and monotonically ordered but touch in some nontrivial region. This is possible to occur: for example  consider a small ball and a large ball merging with a large pinning interval.  This would allow a portion of the free boundary opposite the ``contacting area" to stay fixed as depicted in the simulation \cite[Fig.~2]{FKPii}.  In other words, although strong maximum principle holds for the solutions, it does not hold for the free boundaries.
\end{remark}
\begin{remark}
From \lref{no-crossing} and \partref{lr-limits-jumps} in each connected component $\Lambda$ of $U \cap \{u^*(t) > 0\}$ we have  $\{u^*(t,\cdot)|_\Lambda,u_*(t,\cdot)|_\Lambda\} = \{u_r(t)|_\Lambda,u_\ell(t)|_\Lambda\}$.
\end{remark}
\begin{remark}
$v(t)$ as in \partref{lr-limits-redefine} is always measurable since $v(t) =u(t)$ except at countably many times.
\end{remark}

\begin{proof}[Proof of \pref{lr-limits}]
{\bf Part \partref{lr-limits}.} By \lref{energy-soln-time-reg} the map 
\[t\mapsto (\mathbf{1}_{\Omega(u(t))},\mathcal{J}(u(t)),P(u(t))) \ \hbox{ is in } \ BV([0,T];L^1(\R^d) \times \R \times \R)\] 
and so by \lref{BV-abstract-lr} the map has left and right limits in $L^1(\R^d) \times \R \times \R$ at each $t\in [0,T]$. 

Call $\Omega_\ell(t_0)$ to be the set obtained by the $L^1$ left limit of $\Omega(u(t))$ at $t_0$.  Let $t_n \nearrow t_0$. By uniform Lipschitz continuity of $u(t_n)$ there is a subsequence converging uniformly to some $u_\infty$ which satisfies $u_\infty = F(t_0)$ on $\partial U$. By uniform non-degeneracy of the $u(t)$ we must have $\Omega(u_\infty) = \Omega_\ell(t_0)$. This specifies that $u_\infty$ is the unique solution of
\begin{equation}\label{e.conclude-u-conv-from-omega-conv}
\Delta u_\infty = 0 \ \hbox{ in } \ \Omega_\ell(t_0) \cap U \ \hbox{ with } \ u_\infty = 0 \ \hbox{ on } \ \partial \Omega_\ell(t_0) \ \hbox{ and } \ u_\infty = F(t_0) \ \hbox{ on } \ \partial U.
\end{equation}
 Since the original sequence $t_n \nearrow t_0$ was arbitrary we find that $u(t) \to u_\infty$ uniformly as $t_n \nearrow t_0$. The argument is the same for right limits.
 
 Thus $u(t)$ has left and right limits $u_{\ell/r}(t)$ in the uniform metric at each time $t \in [0,T]$.  Next we need to show the consistency properties \eref{consistency-lr-limits}.  We already showed that $\lim_{s \to_{\ell / r} t} \Omega(u(t)) = \Omega(u_{\ell / r}(t))$ using uniform non-degeneracy in the previous paragraph.   Recalling that $P(u(t)) = F(t)^{-1}\mathcal{D}(u(t))$, to conclude we just need to check the consistency of limits for the Dirichlet energy $\mathcal{D}(u(t))$
 \begin{equation}\label{e.dirichlet-energy-lr-limits}
 \lim_{s \to_{\ell/r} t} \int_U |\grad u(t)|^2 \ dx = \int_U |\grad u_{\ell/r}(t)|^2 \ dx.
 \end{equation}
We claim that $\grad u(s)$ converges locally uniformly on $\Omega(u_{\ell/r}(t))$ to $\grad u_{\ell/r}(t)$ as $s \nearrow t$ (resp. $s \searrow t$). Along with the uniform boundedness of $\grad u(t)$ the dominated convergence theorem gives \eref{dirichlet-energy-lr-limits}. 

As usual we handle the left limit case and the right limit case is similar. Let $K$ be a compact subset of $\Omega(u_\ell(t))$. By the uniform convergence $u(s) \to u_\ell(t)$ as $s \nearrow t$ and by uniform non-degeneracy 
\[ \inf_K u(s) \geq c d(K,\Omega(u_\ell(t))^{\complement})\]
and so by uniform Lipschitz continuity 
\[ d(K,\Omega(u(s))^{\complement}) \geq c d(K,\Omega(u_\ell(t))^{\complement}) \]
for $s < t$ sufficiently close to $t$. Thus the $\grad u(s)$ are uniformly bounded and uniformly continuous on $K$ for $s<t$ sufficiently close to $t$. Typical arbitrary subsequence arguments show that $\grad u(s)$ converges uniformly on $K$ to $\grad u_\ell(t)$ as $s \nearrow t$.

{\bf Part \partref{lr-limits-minimizer}.} We introduce the extraneous notation $u_0(t) = u(t)$ in order to prove minimizer properties for  $u_{\ell} /u_0  / u_r$ at once. 

 First we check that each of $u_{\ell / 0 / r}(t)$ minimize their own respective $\mathcal{E}(\Omega(u_{\ell / 0 / r}(t)),\cdot)$ over $F(t) + H^1_0(U)$. For $u_0(t)$ this is just the global stability property of an energy solution. For $u_\ell(t)$, by the global stability property for any $t_- < t$ we know
\[\mathcal{J}(u(t_-)) \leq J\left(\frac{F(t_-)}{F(t)}v\right) + \textup{Diss}(u(t_-),v) \ \hbox{ for all } \ v \in F(t) + H^1_0(U).\]
Taking the limit as $t_-\nearrow t$, using time continuity of $F$, and applying part \partref{lr-limits}
\[\mathcal{J}(u_\ell(t)) \leq \mathcal{J}(v) + \textup{Diss}(u_\ell(t),v) \ \hbox{ for all } \ v \in F(t) + H^1_0(U).\]
 The proof that $u_r(t)$ minimizes $\mathcal{E}(u_r(t),\cdot)$ is similar, applying global stability of $u(t_+)$ and taking a limit $t_+ \searrow t$.

 Next we check that $u_{0 / r}$ minimize $\mathcal{E}(u_{\ell}(t),\cdot)$ over $F(t) + H^1_0(U)$. Apply the dissipation relation with times $t_- \leq t \leq t_+$
\[ \mathcal{J}(u(t_-))  + \int_{t_-}^{t_+}2\dot{F}(s)P(s) \ ds \geq \textup{Diss}(u(t_-),u(t_+))+\mathcal{J}(u(t_+)).\]
Taking the limit as $t_- \nearrow t$ and $t_+ \searrow t$ and using the left/right continuity of $\mathcal{J}(u(\cdot))$ and $\Omega(\cdot)$ established in part \partref{lr-limits}
\[\mathcal{J}(u_\ell(t)) \geq \mathcal{J}(u_r(t)) + \textup{Diss}(u_\ell(t),u_r(t))\]
or in the case $t_+ = t$ and $t_- \nearrow t$
\[\mathcal{J}(u_\ell(t)) \geq \mathcal{J}(u(t))+  \textup{Diss}(u_\ell(t),u(t))\]
Since we already established that $u_\ell(t)$ is always a minimizer of $\mathcal{E}(u_\ell(t),\cdot)$ also $u_{0 /r}(t)$ must be minimizers.  Finally in the case $t_- = t$ and $t_+ \searrow t$ we obtain similarly
\[\mathcal{J}(u(t)) \geq \mathcal{J}(u_r(t)) + \textup{Diss}(u(t),u_r(t))\]
which gives that $u_r(t)$ minimizes $\mathcal{E}(u(t),\cdot)$.

 Lastly we consider $u_\ell(t)\wedge u_r(t)$ and $u_\ell(t) \vee u_r(t)$.  Since $t$ is fixed for the remainder of this part of the proof we will write $u_\ell = u_\ell(t)$ and $u_r = u_r(t)$ to simplify expressions. 

First of all, since $u_\ell$ and $u_r$ minimize $\mathcal{E}(u_\ell,\cdot)$ over $F(t) + H^1_0(U)$, so also, by \eref{E-additivity-property}, $u_\ell \wedge u_r$ and $u_\ell \vee u_r$ both minimize $\mathcal{E}(u_\ell,\cdot)$ over $F(t) + H^1_0(U)$. 

Then for any $v \in F(t) + H^1_0(U)$
  \[\mathcal{E}(u_\ell, u_\ell \wedge u_r) \leq \mathcal{E}(u_\ell,v)\]
  or
  \begin{align*}
  \mathcal{J}(u_\ell \wedge u_r) &\leq \mathcal{J}(v) + \Diss(u_\ell,v) - \Diss(u_\ell,u_\ell \wedge u_r) \\
  &\leq \mathcal{J}(v) +\Diss(u_\ell \wedge u_r,v)\\
  &= \mathcal{E}(u_\ell\wedge u_r,v)
  \end{align*}
  by the dissipation distance triangle inequality \lref{diss-triangle-eq} for the last inequality.  Thus $u_\ell\wedge u_r$ is globally stable. A similar argument applies to $u_\ell \vee u_r$.
  
  Finally we need to argue that 
  \[u_r \ \hbox{ minimizes $\mathcal{E}(u_\ell\wedge u_r,\cdot)$ and $\mathcal{E}(u_\ell\vee u_r,\cdot)$ over $F(t) + H^1_0(U)$.}\]
We just argue for $\mathcal{E}(u_\ell\wedge u_r,\cdot)$, the other case is similar,
  \begin{align*}
  \mathcal{E}(u_\ell\wedge u_r,v) &=  \mathcal{E}(u_\ell,v)+\textup{Diss}(u_\ell\wedge u_r,v)-\textup{Diss}(u_\ell,v) \\
  &\geq \mathcal{E}(u_\ell,u_r)+\textup{Diss}(u_\ell\wedge u_r,v)-\textup{Diss}(u_\ell,v)\\
  &=\mathcal{J}(u_r) +\Diss(u_\ell,u_r)+ \textup{Diss}(u_\ell\wedge u_r,v)-\textup{Diss}(u_\ell,v)\\
  & = \mathcal{J}(u_r)+\Diss(u_\ell\wedge u_r,v)+\Diss(v,u_r)\\
  &\quad \quad +[\Diss(u_\ell,u_r)-\Diss(u_\ell,v)-\Diss(v,u_r)]\\
  &= \mathcal{J}(u_r)+\Diss(u_\ell\wedge u_r,u_r)\\
  &\quad \quad +(\mu_-+\mu_+)[|\Omega(v) \setminus \Omega( (u_\ell \wedge u_r) \vee u_r)| + |\Omega((u_\ell \wedge u_r) \wedge u_r) \setminus \Omega(v)|]\\
  & \quad \quad -(\mu_-+\mu_+)[|\Omega(v) \setminus \Omega( u_\ell \vee u_r)| + |\Omega(u_\ell \wedge u_r) \setminus \Omega(v)|]\\
  &=\mathcal{E}(u_\ell\wedge u_r,u_r).
  \end{align*}
  where we have used that $u_r$ minimizes $\mathcal{E}(u_\ell,\cdot)$, then the sharp triangle inequality \lref{diss-triangle-eq}, and finally that $(u_\ell \wedge u_r) \wedge u_r = u_\ell \wedge u_r$ and $(u_\ell \wedge u_r) \vee u_r = u_\ell \vee u_r$.

{\bf Part \partref{lr-limits-jumps}.} Since $t$ is fixed in this part of the proof we write $u_{\ell / 0 / r}(t) = u_{\ell / 0 / r}$ dropping the $t$ dependence. By global stability and part \partref{lr-limits-minimizer} respectively we know that $u_\ell$ and $u_r$ both minimize $\mathcal{E}(u_\ell,\cdot)$.  Furthermore \lref{no-crossing} yields that $u_\ell$ and $u_r$ are ordered in each connected component of $\Omega(u_\ell \vee u_r)$.

Next we show that $u_\ell \wedge u_r \leq u \leq u_\ell \vee u_r$.  We use the multiple minimizations in part \partref{lr-limits-minimizer} to show an ``equality in triangle inequality" which implies that $u$ must be in between $u_\ell$ and $u_r$. Intuitively speaking an intermediate jump which crosses either $u_\ell$ or $u_r$ would just incur extra dissipation cost to return to $u_r$.

We compute using part \partref{lr-limits-minimizer}
\begin{align*}
\mathcal{J}(u_r) + \textup{Diss}(u_\ell,u_r)&= \mathcal{J}(u)+\textup{Diss}(u_\ell,u)\\
&=\mathcal{J}(u_r) +\textup{Diss}(u,u_r)+\textup{Diss}(u_\ell,u).
\end{align*}
More specifically the first equality is since $u_r$ and $u$ both minimize $\mathcal{E}(u_\ell,\cdot)$, the second equality is since $u$ and $u_r$ both minimize $\mathcal{E}(u,\cdot)$. Simplifying we find the equality
\begin{equation}\label{e.triangle-equality}
 \textup{Diss}(u_\ell,u)+\textup{Diss}(u,u_r) - \textup{Diss}(u_\ell,u_r) = 0. 
\end{equation}
By the sharp triangle inequality \lref{diss-triangle-eq}
\begin{equation*}\label{e.triangle-inequality-strictness}
\textup{Diss}(u_\ell,u)+\textup{Diss}(u,u_r) - \textup{Diss}(u_\ell,u_r) = (\mu_-+\mu_+) \bigg[ |\Omega(u_\ell \wedge u_r) \setminus \Omega(u)| + |\Omega(u) \setminus \Omega(u_\ell \vee u_r)|\bigg].
\end{equation*}
Together with  \eref{triangle-equality}, the above equality shows that $\Omega(u_\ell \wedge u_r) \subset \Omega(u) \subset \Omega(u_\ell \vee u_r)$. Since $u_\ell \wedge u_r$, $u$, and $u_\ell \vee u_r$ are harmonic in their respective positivity sets by \lref{no-crossing}, then the previous set ordering and maximum principle implies $u_\ell \wedge u_r \leq u \leq u_{\ell} \vee u_r$.

  {\bf Part \partref{lr-limits-redefine}.}  Since $u_\ell(t) = u_r(t) = u(t) = F(t)$ on $\partial U$ the boundary condition is satisfied by $v(t)$. By \partref{lr-limits-minimizer} any $v(t) \in \mathcal U(t)$ is globally stable.

       Finally we need to check the dissipation relation \eref{diss-inequality} for $v$. We do this in two steps, the dissipation on the open interval $(t_0,t_1)$ plus the dissipation at the endpoints.

Fix $t_0 < t_1$ and we claim that
  \begin{equation}\label{e.open-interval-diss-v}
   \mathcal{J}(u_r(t_0))-\mathcal{J}(u_\ell(t_1))  + \int_{t_0}^{t_1}2\dot{F}(s)P(s) \ ds \geq \Diss(u_r(t_0),u_\ell(t_1))
   \end{equation}
  To prove this apply the dissipation relation for $u(t)$ with sequences $t_{0,k}$ and $t_{1,k}$ with $t_{0,k} \searrow t_0$ and $t_{1,k} \nearrow t_1$  \[ \mathcal{J}(u(t_{0,k}))-\mathcal{J}(u(t_{1,k}))  + \int_{t_{0,k}}^{t_{1,k}}2\dot{F}(s)P(s) \ ds \geq  \Diss(u(t_{0,k}), u(t_{1,k})).\]
 Sending $k \to \infty$ and using the continuities from \partref{lr-limits} shows \eref{open-interval-diss-v}.

 Now consider the endpoint dissipations, by \partref{lr-limits-minimizer},
 \[ \mathcal{J}(u_r(t_0))+\Diss(v(t_0),u_r(t_0)) = \mathcal{J}(v(t_0)) \ \hbox{ and } \  \mathcal{J}(v(t_1))+\Diss(u_\ell(t_1),v(t_1)) = \mathcal{J}(u_\ell(t_1)) \]
 since every $v(t_0) \trianglelefteq u_r(t_0)$ for every $v(t_0) \in \mathcal{U}(t_0)$ and $u_\ell(t_1) \trianglelefteq v(t_1) $ for every $v(t_1) \in \mathcal{U}(t_1)$ in the temporal partial ordering (defined in \eref{time-ordering-def}).
 
 Adding these to \eref{open-interval-diss-v} gives
 \begin{align*}
 &\mathcal{J}(v(t_0)) - \mathcal{J}(v(t_1)) + \int_{t_0}^{t_1} 2 \dot{F}(s) P(s) \ ds  \\
 & \quad \quad \geq  \Diss(v(t_0),u_r(t_0))+\Diss(u_r(t_0),u_\ell(t_1))+ \Diss(u_\ell(t_1),v(t_1))\\
 &\quad \quad \geq \Diss(v(t_0),v(t_1))
 \end{align*}
 by applying the triangle inequality \lref{diss-triangle-eq} in the last step. This completes the proof.
\end{proof}

\section{The dynamic slope condition for energy solutions}\label{s.energy-soln-motion-law} 
In this section we show that the energy dissipation balance law implies a weak notion of the dynamic slope condition \eref{dynamic-slope-condition-intro}
\begin{equation*} |\grad u|^2 = 1\pm \mu_\pm \quad \hbox{ if } \quad \pm V_n(t,x) >0.
\end{equation*}
Recall that $V_n(t,x)$ is the outward normal velocity of $\Omega(u(t))$ at $x \in \partial \Omega(u(t))$.

Recall by \pref{lr-limits} part \partref{lr-limits-redefine} if $u$ is an arbitrary energy solution on $U \times [0,T]$ then the envelopes $u^*$ and $u_*$ are also energy solutions and have a simple representation in terms of the left and right limits at each time $u_*(t) = u_\ell(t) \wedge u_r(t)$ and $u^*(t) = u_\ell(t) \vee u_r(t)$.

The main theorem of this section is that if $u(t)$ is an energy solution then $u^*(t)$  and $u_*(t)$ satisfy the dynamic slope condition (up to sets of surface measure zero) in the sub- and supersolution sense respectively.

For the statement we will use the notions of positive and negative velocity based on inner and outer touching space-time cones and the notions of sub and superdifferential, see \fref{cone-pics}. This notion of viscosity solutions is stronger than the usual comparison definition of level set velocity, since we test more free boundary points that have weaker space-time regularity. Nonetheless, in the rate-independent evolution the time variable mostly plays a role of a parameter and so it seems natural to test space and time directions differently. See the companion paper \cite{FKPii} where we proved a comparison principle and regularity results using this notion.

\begin{figure}
\begin{minipage}{.48\textwidth}
\begin{tikzpicture}[scale = 1]
\filldraw[color = gray!40] (-2,0) -- (0,0) .. controls (-.25,-.25) .. (-2,-1);
\filldraw[pattern=vertical lines] (-1,-1) -- (0,0) -- (1,-1);
\draw[black](0,0) .. controls (-.25,-.25) .. (-2,-1);
\filldraw[black] (0,0) circle (.025);
\draw[black] (-1,-1) -- (0,0) -- (1,-1);
\draw[black,dashed] (-2,0) -- (2,0);

\draw[->] (-2.25,-1) -- (-2.25,0);

\node[right] at (2,0) {$t=t_0$};
\node[left] at (-2.25,-.5) {$t$};
\node at (-1.4,-.3) {$\Omega$};
\node[below] at (0,-1) {$|x-x_0| \leq c (t_0-t)$};
\end{tikzpicture}
\end{minipage}
\begin{minipage}{.48\textwidth}
\begin{tikzpicture}[scale = 1]
\filldraw[color = gray!40] (-2,0) -- (0,0) .. controls (.25,-.25) .. (2,-1) -- (-2,-1);
\filldraw[pattern=vertical lines] (-1,-1) -- (0,0) -- (1,-1);
\draw[black](0,0) .. controls (.25,-.25) .. (2,-1);
\filldraw[black] (0,0) circle (.025);
\draw[black] (-1,-1) -- (0,0) -- (1,-1);
\draw[black,dashed] (-2,0) -- (2,0);

\draw[->] (-2.25,-1) -- (-2.25,0);

\node[right] at (2,0) {$t=t_0$};
\node[left] at (-2.25,-.5) {$t$};
\node at (-1.4,-.3) {$\Omega$};
\node[below] at (0,-1) {$|x-x_0| \leq c (t_0 - t)$};
\end{tikzpicture}
\end{minipage}
\caption{Left: velocity $c$ cone touches $\Omega(t)$ from the outside at $(t_0,x_0)$, interpreted as $V_n(t_0,x_0) \geq c$.  Right: velocity $c$ cone touches $\Omega(t)$ from the inside at $(t_0,x_0)$, interpreted as $V_n(t_0,x_0) \leq - c$. }
\label{f.cone-pics}
\end{figure}
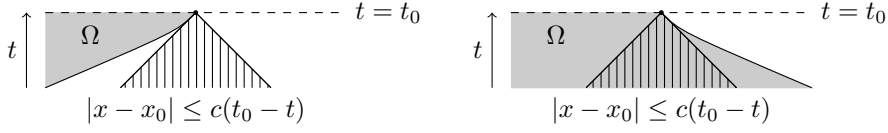

\begin{definition}\label{d.positive-velocity-cone-dynamic-slope}
Given a time varying family of domains $\Omega(t)$ we write for some $x_0 \in \partial \Omega(t_0)$
\begin{align*}
V_n(t_0,x_0) > 0 \qquad \hbox{(resp. $V_n(t_0,x_0) < 0$)}
\end{align*}
 if there are $\eta>0$ and $r_0>0$ such that \begin{equation}\label{cone_positive-dynamic-slope}
\{x: \ |x-x_0| \leq \eta(t_0-t)\} \subset \Omega(t)^{\complement} \qquad \hbox{(resp. $\Omega(t)$)} \ \hbox{ for } \ \ t_0 - r_0 \leq t < t_0.
\end{equation} 
See \fref{cone-pics}.
\end{definition}

\begin{definition}\label{d.supersubdifferential-dynamic-slope}
Given a non-negative continuous function $u$ on $U$, for each $x_0 \in \partial \Omega(u)\cap U$ define
\begin{align*}
D_+ u(x_0) &= \left\{ p \in  \R^d :  u(x) \leq p \cdot (x-x_0) + o(|x-x_0|) \ \hbox{ in } \  \overline{\Omega(u)}\right\}
\intertext{and}
D_- u(x_0) &= \left\{ p \in  \R^d :  u(x) \geq p \cdot (x-x_0) - o(|x-x_0|)\right\}.
\end{align*}
\end{definition}

\begin{theorem}\label{t.ae-viscosity-prop}
 Suppose that $u$ is an upper semicontinuous energy solution \EE{} on $U \times [0,T]$. Then for every $t\in(0,T]$ the set of points where the outward motion law fails
 \[\Gamma^+(u,t) := \{x \in \partial \Omega(t): \ V_n(t,x) > 0, \ |p|^2 < 1+\mu_+ \ \hbox{ for some } \ p \in D_+ u(t,x)\}\]
 has $\mathcal{H}^{d-1}$ measure zero.
  
  Similarly, if $u$ is a lower semicontinuous energy solution \EE{} on $U \times [0,T]$, then for every $t\in(0,T]$ the set of points where the inward motion law fails
    \[\Gamma^-(u,t) := \{x \in \partial \Omega(t): \ V_n(t,x) < 0, \ |p|^2 > 1-\mu_- \ \hbox{ for some } \ p \in D_- u(t,x)\}\]
 has $\mathcal{H}^{d-1}$ measure zero.
\end{theorem}

We will only prove the result for $\Gamma^+$. The proof for $\Gamma^-$ is analogous.

\begin{remark}
The sub and superdifferential may both be empty at some free boundary points. However, for energetic solutions, we can follow standard blow-up arguments (see \cite[Section 3.3]{CaffarelliSalsa}) to show that the sub and super-differential are both non-trivial $\mathcal{H}^{d-1}$-a.e. on $\partial \Omega(u(t))$. 

Recall that $\partial \Omega(u(t))$ is a finite perimeter set and $\partial \Omega(u(t))$ has finite $\mathcal{H}^{d-1}$ measure (\lref{initial-fb-regularity}) for each $t>0$. Therefore the reduced boundary $\partial_*\{u(t)>0\}$ has full $\mathcal{H}^{d-1}$ measure on $\partial \{u(t)>0\}$. 

The sub and superdifferentials are both non-trivial at every point of the reduced boundary.  More precisely, at $x_0 \in\partial_*\{u>0\}$ with measure theoretic normal $n_0$ consider the blow-up sequence
\[u_r(x) = \frac{u(x_0+rx)}{r}\]
which are uniformly bounded, uniformly non-degenerate, and uniformly Lipschitz. All subsequential blow up limits (not necessarily unique) must then be viscosity solutions of
\[\Delta v = 0 \ \hbox{ in } \ \{x \cdot n_0>0\} \ \hbox{ and } \ 1-\mu_- \leq |\grad v|^2 \leq 1+\mu_+ \ \hbox{ on } \ \partial \{x \cdot n_0 >0\}.\]
From this one can conclude that for any (subsequential, locally uniform) blow up limit
\[(1-\mu_-)^{1/2}(x \cdot n_0)_+ \leq v(x) \leq (1+\mu_+)^{1/2}(n_0\cdot x)_+\]  
and this implies that at least
\[(1+\mu_+)^{1/2} n_0 \in D_+u(x_0) \ \hbox{ and } \ (1-\mu_-)^{1/2} n_0 \in D_-u(x_0),\]
i.e. the sub and superdifferential are nontrivial on the measure theoretic reduced boundary.
\end{remark}

Before we proceed to the details let us give a description of the proof, which also explains why we can only prove the solution condition in an almost everywhere sense.

 We start by presenting the formal argument, valid in the case that everything is $C^1$ in space and time. Computing the time derivative directly and integrating by parts
\[ \frac{d}{dt}\mathcal{J}(u(t)) = \int_{\partial \Omega(t)} (1-|\grad u|^2) V_n \ dS +2\dot{F}(t)P(t)\]
but also differentiating the energy dissipation balance from \lref{energy-diss-eq} yields
\[ \frac{d}{dt}\mathcal{J}(u(t)) = 2 \dot{F}(t) P(t) - \int_{\partial \Omega(t)} \mu_+(n) (V_n)_+ + \mu_-(n) (V_n)_- \ dS.\]
Combining the above two identities we find
\begin{equation}\label{e.formal-energy-argument-conclusion}
 \int_{\partial \Omega(t)} (1+\mu_+(n)-|\grad u|^2) (V_n)_+ + (|\grad u|^2 - 1 + \mu_-(n))(V_n)_- \ dS = 0. 
 \end{equation}
 Both terms in the above integral are non-negative (Corollary~\ref{c.energy-stability-implies-viscosity-soln}), and so they must actually be zero pointwise.

Of course we cannot exactly use the formal argument. Even without jumps it is tricky to justify taking a time derivative.  Instead we need to make a similar kind of energy argument with discrete differences.  Furthermore, even if we could justify the identity \eref{formal-energy-argument-conclusion}, without continuity of $\grad u$ we could only derive the slope condition in the surface measure a.e. sense. Recall that we do not know  $C^1$ regularity of general energy solutions. The regularity theory of \cite{FKPii} applies only to viscosity solutions and we are trying to prove a viscosity solution type property. This is why we can only achieve the dynamic slope condition in the almost everywhere sense.

Although we do need to deal carefully with the jumps, the situation is actually better when the free boundary jumps: in this case we can guarantee that slope condition is satisfied pointwise \emph{everywhere}. 
\begin{corollary}\label{c.jump-subsoln-cond}

Let $F>0$ and suppose that $u\geq u_\ell$ (or $u\leq u_\ell$) is a minimizer of
\begin{equation}\label{energy_above}
\mathcal{E}(u_\ell,u) = \mathcal{J}(u) + \textup{Diss}(u_\ell,u) \ \hbox{ over } \ F+H^1_0(U).
\end{equation}
Then for any $x_0 \in \partial \Omega(u) \setminus \overline{\Omega(u_\ell)}$ (or $x_0 \in \partial\Omega(u) \cap \Omega(u_\ell)$)
\[|p|^2 \geq  (1+\mu_+) \  \hbox{ for all } \  p \in D_+u(x_0) \  \ \ (\hbox{or } |p|^2\leq (1-\mu^-)  \  \hbox{ for all } \  p \in D_-u(x_0)).\]

\end{corollary}
\begin{proof}
Apply \lref{inward-outward-min-props} and \lref{inward-outward-min-implies-vs}.
\end{proof}

In order to make the main ideas of the proof of \tref{ae-viscosity-prop} more clear we will state in a separate lemma the construction of the energy competitor.

First we explain the idea of the energy competitor and then give a precise statement in \lref{barrier-construction} below. Suppose that the advancing part of the dynamic slope condition fails at some time $t_0$ on a set of positive $\mathcal{H}^{d-1}$ measure. Then for times $t_{-1} = t_0 - \delta$ sufficiently close to $t_0$, we can perform an inward perturbation of $u(t_0)$ in the normal direction in a region with measure \emph{at least} $O(\delta)$, and without crossing $u(t_{-1})$.  This perturbation will also have slope strictly smaller than the pinning interval endpoint value $(1+\mu_+)^{1/2}$ in the perturbation region. This allows us to estimate the change in the Dirichlet energy.

\begin{lemma}\label{l.barrier-construction}
 Let $u$ be upper semi-continuous on $[0,T]$, harmonic in $\Omega(u(t))$ at each time, and satisfying the conclusions of \lref{initial-fb-regularity} (e.g. $u$ an upper semicontinuous energy solution). 
 
 Suppose for some $t_0 \in (0,T]$ 
\[\mathcal{H}^{d-1}\left(\bigg\{x \in \partial \Omega(t_0): \ V_n(t_0,x) > 0, \ |p|^2 < 1+\mu_+ \ \hbox{ for some } \ p \in D_+ u(t_0,x)\bigg\}\right) > 0. \] 
Then there exists $c>0$ such that the following holds: for all sufficiently small $\delta = t_0 - t_{-1} >0$  there is a function  $u_{\delta}\in H^1(U)$  with $u_\delta\geq 0$ and $u_\delta = F(t_0)$ on $\partial U$ such that 
\begin{equation}\label{e.perturbation-properties}
 \Omega(u(t_{-1})) \subset \Omega(u_\delta) \subset \Omega( u(t_0)), \ |\Omega(u(t_0)) \setminus \Omega(u_\delta)| \geq c\delta
 \end{equation}
and
\begin{equation}\label{e.dirichlet-energy-barrier-ineq}
\int_{\Omega(u_\delta)} |\grad u_\delta|^2 \ dx - \int_{\Omega(t_0)} |\grad u(t_0)|^2 \ dx \leq (1+\mu_+ - c)|\Omega(u(t_0)) \setminus \Omega(u_\delta)|.
\end{equation}
\end{lemma}

\begin{remark}
The constant $c>0$ in above Lemma is not very easily quantified, it depends delicately on the hypothesis that $\mathcal{H}^{d-1}(\Gamma^+(u,t_0))>0$ (not just on the measure itself, but on the consequences of being positive measure).
\end{remark}

We defer the proof of  \lref{barrier-construction} to the end of the section.  It is worth emphasizing that this part is the technical heart of the proof of \tref{ae-viscosity-prop}.

Now we give the proof of the weak solution property \tref{ae-viscosity-prop} via an energy argument using \lref{barrier-construction}.  \pref{lr-limits} also plays a central role allowing us to reduce to the case of left continuous times.

\begin{proof}[Proof of \tref{ae-viscosity-prop}]
Fix $t_0 \in (0,T]$, we first assume that $P(u(t))$ is left continuous at $t_0$.  Jump times will be considered at the end, building on the result at left continuous times and using the upper semi-continuity of $u$ as well as the other regularity properties of energy solutions from \pref{lr-limits}. 

{\bf Step 1.} Assume for now that
\[|P(u(t_0)) - P(u(t))| \leq \sigma(t_0 - t) \ \hbox{ for } \ t \leq t_0\]
with some modulus of continuity $\sigma$. Suppose that
\[\mathcal{H}^{d-1}\left(\Gamma^+(u,t)\right) > 0. \] 
Then \lref{barrier-construction} yields $u_\delta$, $t_{-1},t_0$ for small $\delta>0$.  We define
\[ v(x) := \frac{F(t_{-1})} {F(t_0)}u(t_0,x) \ \hbox{ and } \ v_\delta(x) := \frac{F(t_{-1})} {F(t_0)} u_\delta(x).\]

From the energy dissipation inequality \eref{diss-inequality} evaluated on $[t_{-1},t_0]$, we derive
\begin{align}
\textup{Diss}(\Omega(t_{-1}),\Omega(t_0)) & \leq \mathcal{J}(u(t_{-1})) - \mathcal{J}(u(t_0)) + \int_{t_{-1}}^{t_0} 2 \dot{F}(t)P(t) \ dt\notag\\
&=[\mathcal{J}(u(t_{-1})) - \mathcal{J}(v_\delta(t_0))]+[\mathcal{J}(v_\delta(t_0)) - \mathcal{J}(v(t_0))] \notag\\
&\quad +\left[\mathcal{J}(v(t_0)) - \mathcal{J}(u(t_0))+ \int_{t_{-1}}^{t_0} 2 \dot{F}(t)P(t) \ dt\right] \notag \\
&= A+B+ C.\label{e.energy-diss-expand-udelta}
\end{align}
We next explain, individually, how to bound $A,B,C$ to yield a contradiction.

Since $v_\delta = F(t_{-1})$ on $\partial U$, we can apply the stability property \dref{energy_solution}\eqref{e.stability} at time $t_{-1}$ to obtain
\[A=\mathcal{J}(u(t_{-1})) - \mathcal{J}(v_\delta) \leq \textup{Diss}(\Omega(t_{-1}),\Omega_\delta(t_0)).  \]

Next the energy inequality \eref{dirichlet-energy-barrier-ineq} from \lref{barrier-construction} yields
\[B:=\mathcal{J}(v_\delta(t_0)) - \mathcal{J}(v(t_0)) \leq \left( \frac{F(t_{-1})^2} {F(t_0)^2}(1+\mu_+ - c) - 1 \right)|\Omega(t_0) \setminus \Omega_\delta(t_0)|.\]
Recall that $c>0$ is independent of $\delta>0$, so from the Lipschitz continuity of $\log F(t)$ we conclude that
\[B \leq (\mu_+-c + o_\delta(1))|\Omega(t_0) \setminus \Omega_\delta(t_0)|.\]

Lastly we address $C$. We have
\begin{align}
\mathcal{J}(u(t_0)) - \mathcal{J}(v(t_0))&= \int_U (1-\frac{F(t_{-1})}{F(t_0)} ) |Du|^2(t_0)  dx \notag\\
&= \frac{F^2(t_0)-F^2(t_{-1})}{F(t_0)} P(t_0)  \notag\\
&\geq \int_{t_{-1}}^{t_0} 2 \dot{F}(t)P(t)  dt - 2\|\dot{F}\|_{L^\infty}\sigma(t_0 - t_{-1})|t_0-t_{-1}|\notag,\\
&=  \int_{t_{-1}}^{t_0} 2 \dot{F}(t)P(t)  dt - o_\delta(1)\delta. \label{e.deriv-hits-bdrycond}
\end{align}
This is the key place where we use the left continuity hypothesis on $P$.

Now, combining the previous three inequalities into \eref{energy-diss-expand-udelta}, we have obtained
\[ \textup{Diss}(\Omega(t_{-1}),\Omega(t_0)) \leq \textup{Diss}(\Omega(t_{-1}),\Omega_\delta(t_0)) + (\mu_+ -c+ o_\delta(1))|\Omega(t_0) \setminus \Omega_\delta(t_0)| + o_\delta(1) \delta  \]
On the other hand, since $\Omega(u(t_{-1})) \subset \Omega_\delta(t_0) \subset \Omega(t_0)$ from \lref{barrier-construction}, we have
\begin{align*}
 \textup{Diss}(\Omega(t_{-1}),\Omega(t_0)) -  \textup{Diss}(\Omega(t_{-1}),\Omega_\delta(t_0)) &=  \mu_+|\Omega(t_0) \setminus \Omega_\delta(t_{-1})|- \mu_+|\Omega_\delta(t_0) \setminus \Omega(t_{-1})|\\
 &=\mu_+|\Omega(t_0) \setminus \Omega_\delta(t_0)|.
\end{align*}
Putting the above estimates together, it follows that
\[ \mu_+|\Omega(t_0) \setminus \Omega_\delta(t_0)| \leq (\mu_+ -c+ o_\delta(1))|\Omega(t_0) \setminus \Omega_\delta(t_0)| + o_\delta(1) \delta.\]
Since \lref{barrier-construction} crucially guarantees that $|\Omega(t_0) \setminus  \Omega_{\delta}(t_0)| \geq c \delta$, we can divide through by $|\Omega(t_0) \setminus \Omega_\delta(t_0)|$ and find
\[
\mu_+ \leq   \mu_+ -c + o_\delta(1)
\]
Sending $\delta  = (t_0 - t_{-1})$ to zero yields $ 0 \leq - c$, which is a contradiction.

{\bf Step 2.} (Jump times) Suppose that $P(u(t))$ is not left continuous at $t_0$.  By \pref{lr-limits} \partref{lr-limits} $u$ has left and right limits at time $t_0$, $u_\ell(t_0)$ and $u_r(t_0)$. Furthermore, by \pref{lr-limits} part \partref{lr-limits-jumps} and since $u$ is upper semicontinuous we must have $u(t_0) = u_\ell(t_0) \vee u_r(t_0)$.

By \pref{lr-limits} part \partref{lr-limits-redefine}
\[\tilde{u}(t) := \begin{cases} u(t) & t< t_0 \\
u_\ell(t_0) & t = t_0.
\end{cases}\]
is itself an energy solution on $[0,t_0]$ and $\tilde{u}(t)$ and $P(\tilde{u}(t))$ are left continuous at $t_0$.  The previous arguments give
\[\mathcal{H}^{d-1}(\Gamma^+(\tilde{u},t_0)) = 0.\]

We aim to show $\Gamma^+(u,t_0) \subset \Gamma^+(\tilde{u},t_0)$ which will complete the proof showing $\mathcal{H}^{d-1}(\Gamma^+(u,t_0)) = 0$. By \cref{jump-subsoln-cond}
\[\Gamma^+(u,t_0) \setminus \partial \Omega(u_\ell(t_0)) = \emptyset.\]
On the other hand since $u$ touches $\tilde{u}$ from above, at each $x_0 \in \partial \Omega_\ell(t_0) \cap \partial \Omega(t_0)$
\[V_n(t_0,x_0; u) > 0 \hbox{ implies } V_n(t_0,x_0; \tilde u) > 0 \hbox{,\quad and } \ D_+u(t_0,x_0) \subset D_+\tilde{u}(t_0,x_0).\]
Thus
\[\Gamma^+(u,t_0) \subset \Gamma^+(\tilde{u},t_0),\]
which concludes the proof.
\end{proof}

\begin{remark}
The regularity of the pressure $P(t)$ seems to play an important role.  In \eref{deriv-hits-bdrycond} the regularity of $P(t)$ dictates the size of a positive error term which needs to be balanced by the positive $\mathcal{H}^{d-1}$ measure of the set where the viscosity solution condition fails.
\end{remark}

Before continuing to the proof of the barrier construction \lref{barrier-construction}, it is useful to separate out one further measure-theoretic Lemma which says that if the subsolution condition fails on a set of positive surface measure, then it fails so quantitatively.

Let us first define quantified versions of the velocity and sub/superdifferentials.

Here we consider $\Omega(t) = \Omega(u(t))$.
If for given $r_0, \eta > 0$ \eqref{cone_positive-dynamic-slope} holds with $\Omega(t)^\complement$ (resp. $\Omega(t)$) we write
\[V_n^{r_0}(t_0,x_0) \geq \eta \qquad \hbox{(resp. $\leq -\eta$)}.\]

For given $r_0, \sigma >0$ we define
\[D^{\sigma,r_0}_+ u(x_0) := \left\{ p \in  \R^d : u(x) \leq p \cdot (x-x_0) + \sigma|x-x_0| \ \hbox{ in } \ B_{r_0}(x_0) \cap \overline{\{u>0\}} \right\} , \]
and 
\[D^{\sigma,r_0}_- u(x_0) := \left\{ p \in  \R^d : u(x) \geq p \cdot (x-x_0) - \sigma|x-x_0| \ \hbox{ in } \ B_{r_0}(x_0) \cap \overline{\{u>0\}} \right\} .\]
 Lastly we define a quantified set where the dynamic slope condition strictly fails:
\begin{equation}
\Gamma^+_{\eta,\sigma,r_0}(u,t) := \left\{ x \in \partial \Omega(u(t)): \begin{array}{cc}V^{r_0}_n(t,x) \geq\eta \ \hbox{ and } \  \exists p \in D_+^{\sigma,r_0} u(t,x) \\ \hbox{s.t.} \ |p|^2 - (1+\mu_+)  < -\eta\end{array}\right\}.
\end{equation}
The subdifferential version is defined similarly. In the above definition we are keeping more parameters than minimally necessary in order to clarify certain delicate points in the proof.

\begin{lemma}\label{l.monotone-convergence}
The family of sets $\Gamma^+_{\eta,r_0,\sigma}(u,t)$ is monotone decreasing with respect to $\eta$ and $r_0$ and monotone increasing with respect to $\sigma$.  Furthermore
\[\Gamma^+(u,t) = \bigcup_{\eta>0}\bigcap_{\sigma>0}\bigcup_{r_0>0} \Gamma^+_{\eta,\sigma,r_0}(u,t)\]
and thus, by monotone convergence theorem, if $\mathcal{H}^{d-1}(\Gamma^+(u,t))>0$ then there is $\eta>0$ so that for all $\sigma>0$ there is $r_0(\sigma)>0$ such that $\mathcal{H}^{d-1}(\Gamma^+_{\eta,\sigma,r_0}(u,t))>0$.
\end{lemma}
\begin{proof}

The monotonicities follow from the set definition.
Let us check the set formula. Let $x \in \Gamma^+(u,t)$, which means
\[V_n(t,x) > 0 \ \hbox{ and there is $p \in D_+u(t,x)$ with } \ |p|^2 < 1+\mu_+.\]
 By definition of $D_+u$
\[u(t,x) \leq p \cdot x + o(|x-x_0|) \ \hbox{ for } \ x \in \overline{\Omega(u(t))}\]
and so for every $\sigma>0$ there is $r_1(\sigma)>0$ sufficiently small so that
\[u(t,x) \leq p \cdot x + \sigma |x-x_0| \ \hbox{ for } \ x \in \overline{\Omega(u(t))} \cap B_{r_1}(x)\]
Namely $p \in D_+^{\sigma,r_1}u(t,x)$. By the definition of $V_n$
\[V^{r_2}_n(t,x) \geq \eta_2 > 0 \ \hbox{ for some } \ r_2,\eta_2>0. \]
Call $\eta_1 := (1+\mu_+) - |p|^2 >0$. Taking $r_0 = \min\{r_1,r_2\}$ and $\eta = \min\{\eta_1,\eta_2\}$ we find that $x \in \Gamma^+_{\eta,\sigma,r_0}(u,t)$.
\end{proof}

\begin{proof}[Proof of \lref{barrier-construction}]
 By \lref{monotone-convergence} and the hypothesis of the Lemma there is $\eta>0$ such that for all $\sigma>0$ there is positive $r_0(\sigma)$ with
 \begin{equation}\label{e.positive-measure-eta-sigma-r0}
 \mathcal{H}^{d-1}(\Gamma^+_{\eta,\sigma,r_0}(u,t_0)) \geq \frac{1}{2}\mathcal{H}^{d-1}(\Gamma^+(u,t_0))>0.
 \end{equation}
  We will take $\sigma = c(d)\eta$, where the dimensional constant $c(d)$ will be determined in the course of the proof (under \eqref{e.a1-a1-cond1}). This choice of $\sigma$ also fixes $r_0 = r_0(d,\eta)>0$, we will usually drop the dependence on the dimension when it is not important and write (for example) $r_0(\eta)$. Let us consider $\delta \in (0, r_0(\eta)]$ and $t_{-1} := t_0 - \delta$ for the rest of the proof.   By monotonicity of the sets $\Gamma^+_{\eta,\sigma,r_0}(u,t_0)$ in $\eta$, we may assume that $\eta < \frac{1}{2}$.

We outline the proof, which is divided into three steps.  In the first step we use the speed and slope condition at a single point to construct a strict supersolution replacement in a local neighborhood which is below $u(t_0)$ by $O(\delta^d)$ in measure but still above $u(t_{-1})$.   The construction is fairly standard, using the first order asymptotic information from the quantified superdifferential, sliding the linearization inward while also bending up in the tangential directions and bending down in the normal direction in order to create a local superharmonic perturbation. In the second step we perform a typical Vitali-type covering argument.  In the final step we perform the local supersolution replacement in each disjoint ball from the covering to create the perturbation $u_\delta$, which is now $O(\delta)$ below $u(t_0)$ because we have shifted inward by $O(\delta)$ on a set of positive surface measure.

{\bf Step 1.} (Construction of a barrier based on single reference point)  In order to make certain parameter choices more clear in the proof below we will adopt the notation $\eta_t = \eta_x = \eta$ and use $\eta_t$ when the lower bound of the velocity is being used, and $\eta_x$ when the lower bound of the slope condition is being used. Let $x_0\in \Gamma^+_{\eta,\sigma,r_0}(u,t_0)$. For simplicity we may assume that  $x_0 = 0$ and $p_0 = \alpha e_d$ for some $\alpha \geq 0$, so that
\begin{equation}\label{e.u-upper-exp-alpha}
u(t_0,x) \leq \alpha x_d + \sigma|x| \ \hbox{ in } \ B_{r_0}(0) \cap \overline{\Omega(t_0)}
\end{equation}
and, since we can always increase $\alpha$ if necessary, we may assume that
\[\alpha  = (1+\mu_+)^{1/2} -\eta_{x}.\]
  Note that, since $u$ is nonnegative,  \eref{u-upper-exp-alpha} yields that $\alpha (x_d)_- \leq \sigma r \hbox{ on } \ B_{r}(0) \cap \overline{\Omega(t_0)}$, namely
 \begin{equation}\label{e.xd-neg-part-bd}
B_{r}(0) \cap \overline{\Omega(t_0)} \subset \{x_d \geq - \frac{\sigma}{\alpha}r\} \ \hbox{ for all } \ r \leq r_0.
\end{equation}

Furthermore, again from the definition of $\Gamma^+_{\eta,\sigma,r_0}(u,t_0)$, 
\begin{equation}\label{e.positive-speed-outcome}
B_{\eta_t(t_0-t)}(0) \subset \R^d \setminus \Omega(t) \ \hbox{ for } \ t \geq t_0 - r_0.
\end{equation}
Recall that we assumed that $\delta = t_0 - t_{-1}$ is smaller than $r_0(\eta)>0$.

Our perturbation will be performed in $e_d$-aligned cylinders which are anisotropic, slightly wider in the tangential directions than the normal direction. Denoting $x' = (x_1,\dots,x_{d-1})$ and $B_r'$ the ball of radius $r$ in $\R^{d-1}$,
\[\textup{Cyl}_{r}  = B_r' \times (-c_dr,c_dr).\]
 The dimensional constant $c_d := \frac{1}{\sqrt{2(d-1)}}$ is chosen so that on the side boundary of the cylinder $\partial_{side}\textup{Cyl}_{r}:= \partial B_{r} \times (-c_dr,c_dr)$
 \begin{equation}\label{e.cylinder-fatness-choice}
 |x'|^2 - (d-1)x_d^2 \geq r^2 - (d-1)c_d^2 r^2 \geq \frac{1}{2}r^2 \ \hbox{ on } \ x \in \partial_{side}\textup{Cyl}_{r}.
 \end{equation}
 We also call 
 \[\partial_{top}\textup{Cyl}_{r}:=  B_{r} \times \{x_d = c_dr\} \ \hbox{ and } \  \partial_{bot}\textup{Cyl}_{r}:=  B_{r} \times \{x_d = -c_dr\}.\]
 Obviously $\textup{Cyl}_r \subset B_{r_0}$ for $r \leq r_0/\sqrt{2}$.

  We aim to make a localized perturbation in $\textup{Cyl}_{\frac{1}{\sqrt{2}}\eta_t \delta}$. Note that this cylinder is chosen since, by \eref{positive-speed-outcome},
  \begin{equation}\label{e.positive-speed-outcome2}
  \textup{Cyl}_{\frac{1}{\sqrt{2}}\eta_t \delta} \subset B_{\eta_t \delta} \subset \R^d \setminus \Omega(t_{-1}).
  \end{equation}
 Since we assume $\delta = t_0 - t_{-1} \leq r_0$ and $\eta<1$, so $\eta_t \delta \leq r_0$ as well.

   We shift the plane $\alpha x_d$ inward by $O(\delta)$ in the normal direction $e_d$ while also bending upwards in the tangential directions so that the barrier is below $u(t_0)$ but above $u(t_{-1})$.

   More precisely, define the barrier
\begin{equation}
 \psi_{\delta}(x) := (\alpha+\tfrac{\eta_x}{4}) (x_d - a_1\eta_t\delta ) + a_2(\eta_t\delta)^{-1}(|x'|^2 - (d-1)x_d^2).
 \end{equation}
 We will see that if we choose $a_i := c(d,\mu_+) \eta_x, i=1,2$ then we can guarantee the following important properties of $\psi_\delta$:
 \begin{enumerate}[label = (\roman*)]
 \item\label{part.harmonic-parabola} $\psi_\delta$ is harmonic;
 \item\label{part.bends-up} $\psi_\delta(x) > u(t_0,x)$ on  $\partial \textup{Cyl}_{\frac{1}{\sqrt{2}}\eta_t \delta} \cap \overline{\Omega(t_0)}$;
\item\label{part.gradient-bound} $|\grad \psi_\delta(x)|^2 \leq 1 + \mu_+ - \tfrac{1}{2}\eta_x \ \hbox{ in } \textup{Cyl}_{\frac{1}{\sqrt{2}}\eta_t \delta} $;
\item\label{part.exterior-ball} There is $c>0$ depending on $d$ and $\mu_+$ such that 
\[B_{c\eta_x\eta_t\delta}(0) \subset \{\psi_\delta \leq 0\}.\]
 \end{enumerate}
 
 Property \partref{harmonic-parabola}. This is straightforward, but we do emphasize that the superharmonicity is important and needed later to deal with the energy computation in Step 3 below.   

Property \partref{bends-up}. We check the bottom, sides and top of the cylinder separately.  On the top boundary $x \in \partial_{top}\textup{Cyl}_{\frac{1}{\sqrt{2}}\eta_t\delta}$ we can use the extra room on the slope: \begin{align*}
\psi_\delta(x) &\geq \alpha x_d + c(d)\eta_x\eta_t \delta - C(d)(a_1+a_2)\eta_t \delta\\
&\geq \alpha x_d + c(d)[\eta_x - C(d)(a_1+a_2)]\eta_t \delta.
\end{align*}
Thus if we take 
\begin{equation}\label{e.a1-a1-cond1}
a_1 +a_2 \leq c(d)\eta_x
\end{equation} then, again on $x \in \partial_{top}\textup{Cyl}_{\frac{1}{\sqrt{2}}\eta_t\delta}$,
\begin{align*}
\psi_\delta(x) &\geq \alpha x_d + c(d)\eta_x \eta_t \delta \\
&\geq \alpha x_d+c(d)\eta_x |x|  \geq \alpha x_d + \sigma |x| \geq u(t_0,x).
\end{align*}
where this computation fixes the choice of the dimensional constant in $\sigma = c(d)\eta_x$.

On the sides $\partial_{side}\textup{Cyl}_{\frac{1}{\sqrt{2}}\eta_t\delta} \cap \overline{\Omega(t_0)}$ we can use the upwards bending in the tangential directions. Using \eref{cylinder-fatness-choice}, we have.
\begin{align*}
\psi_\delta(x) &\geq \alpha x_d -\tfrac{1}{4}\eta_x (x_d)_--(1+\mu_+)^{1/2}a_1\eta_t\delta + \frac{1}{2}a_2\eta_t \delta  \\
&= \alpha x_d + \sigma |x| +\underbrace{ \left[\tfrac{1}{2}a_2\eta_t \delta - \sigma |x|-\tfrac{1}{4}\eta_x (x_d)_- - (1+\mu_+)^{1/2}a_1\eta_t \delta \right] }_\text{$:=A$}.
\end{align*}
 \eref{xd-neg-part-bd} yields that 
  \[ A \geq c\left[a_2 - C(d)\sigma - C(\mu_+,d)a_1 \right] \eta_t \delta \hbox{ on }  x \in \partial_{side}\textup{Cyl}_{\frac{1}{\sqrt{2}}\eta_t\delta} \cap \overline{\Omega(t_0)}.
   \]
We want this to be non-negative, so along with condition \eref{a1-a1-cond1}  we need to choose $\sigma$, $a_1$ and $a_2$ so that
\[a_1 + a_2 \leq c(d)\eta_x \ \hbox{ and } \ C(d)\sigma + C(\mu_+,d)a_1 \leq a_2.\]
To satisfy both inequalities, we  choose $\sigma = c(d) \eta_x$, $a_1 = c(d,\mu_+) \eta_x$ and $a_2 = c(d)\eta_x$.  Applying these choices of parameters, we find
\[\psi_\delta(x) \geq \alpha x_d + \sigma |x| \geq u(t_0,x) \ \hbox{ on } \ \partial_{side}\textup{Cyl}_{\frac{1}{\sqrt{2}}\eta_t\delta} \cap \overline{\Omega(t_0)}.\]

Finally, the bottom boundary $\partial_{bot}  \textup{Cyl}_{\frac{1}{\sqrt{2}}\eta_t\delta}$ is compactly contained in the zero level set of $u(t_0)$. Namely the set $(x_d)_- = c_d\eta_t \delta$ does not intersect $\overline{\Omega(t_0)}$ by \eref{xd-neg-part-bd} and fixing the choice of dimensional constant in the original specification $\sigma = c(d) \eta$.

Property \partref{gradient-bound}. For this we just compute the gradient and bound with triangle inequality \begin{align*}
|\grad \psi_\delta- (\alpha + \tfrac{1}{4}\eta_x)e_d|  \leq C(d)a_2 \quad \hbox{ in } \textup{Cyl}_{\frac{1}{\sqrt{2}}\eta_t\delta}.
\end{align*}
Choosing the dimensional constant in $a_2 = c(d)\eta_x$ smaller if necessary, 
\[|\grad \psi_\delta| \leq \alpha + \tfrac{1}{2}\eta_x \leq (1+\mu_+)^{1/2} - \tfrac{1}{2}\eta_x \quad \hbox{ in } \textup{Cyl}_{\frac{1}{\sqrt{2}}\eta_t\delta}.    \]

Property \partref{exterior-ball}.    Note that  $\psi_\delta$ is monotone in $\textup{Cyl}_{\frac{1}{\sqrt{2}}\eta_t\delta}$ on the cone of directions, namely for the directions $e \in S^{d-1}$ such that   $e \cdot e_d \geq \frac{1}{2} \geq C(d)a_2$.
Since  $\psi_\delta(a_1\eta_t\delta e_d) = 0$ we find that
\[B_{2a_1\eta_t\delta e_d}(0) \subset \{\psi_\delta \leq 0\}.\]

{\bf Step 2.} (Covering setup) Since Hausdorff measure is inner regular, there is a compact subset $K \subset \Gamma^+_{\eta,\sigma,r_0}(u,t_0)$ with
\[\mathcal{H}^{d-1}(K) \geq \frac{1}{2}\mathcal{H}^{d-1}(\Gamma^+_{\eta,\sigma,r_0}(u,t_0)) \geq \frac{1}{4}\mathcal{H}^{d-1}(\Gamma^+(u,t_0)).\]
Since $K$ is compact, there is $r_1>0$ such that
\begin{equation}\label{e.hausdorff-def}\inf\{ \sum_{j=1}^\infty \rho_j^{d-1} : K \subset \bigcup_{j=1}^\infty B_{\rho_j}(x_j), \ \rho_j \leq r_1\} \geq \frac{1}{2}\mathcal{H}^{d-1}(K).
\end{equation}

Let $3\eta_t\delta \leq r_1$.  By the Vitali covering lemma there is a finite collection of disjoint balls $B_j$ of radius $\eta_t\delta$ centered at points $x_j \in K \subset \partial \Omega(t_0)$ with
\[K \subset \bigcup_j 3B_j.\]
Call $\tilde{B}_j = B_{c\eta_x\eta_t \delta}(x_j)$ with constant $1>c(d,\mu_+)>0$ from property \partref{exterior-ball} above. Since $\tilde{B}_j \subset B_j$ they are a disjoint collection as well.  By \eref{hausdorff-def}
\[\sum_j 3^{d-1} \delta^{d-1} \geq \frac{1}{2}\mathcal{H}^{d-1}(K).\]
By measure non-degeneracy of the free boundary \lref{initial-fb-regularity}
\[|\Omega(t_0) \cap \tilde{B}_j| \geq c \eta^{-2d}\delta^{d},\]
where the constant $c$ depends on $d$ and $\mu_+$.  

Since $\tilde{B}_j's$ are disjoint, we obtain the following by absorbing universal constants into $c$:
\[
|\Omega(t_0) \cap \bigcup_j \tilde{B}_{j}| = \sum_j |\Omega(t_0) \cap \tilde{B}_{j}| \geq \sum_j c \eta^{-2d}\delta^{d} \geq c\eta^{-2d}\mathcal{H}^{d-1}(K) \delta,
\]
or,
\begin{equation}\label{e.measure-diff-lower-bound}
|\Omega(t_0) \cap \bigcup_j \tilde{B}_{j}|  \geq c\eta^{-2d}\mathcal{H}^{d-1}(\Gamma^+(u,t_0)) \delta.
\end{equation}

{\bf Step 3.} (Construction of the comparison functions) 
Let $\delta \leq \min\{ (3\eta_t)^{-1}r_1,r_0\}$ so that both Step 1 and Step 2 apply, recall $r_1$ was fixed by \eref{hausdorff-def} and $r_0(\eta)$ was fixed at the beginning of the proof below \eref{positive-measure-eta-sigma-r0}. Now for each $x_j \in \Gamma^+(u,t_0)$ from the covering constructed in Step 2, call $\psi_{x_j,\delta}$ to be the single-point barrier function constructed in Step 1. The $\psi_{x_j,\delta}$ are defined in $B_j$, extend them to be equal to $+\infty$ outside of their respective $B_j$. 

Consider the comparison function
\[ u_\delta(x)  = 
u(t_0,x) \wedge \bigwedge_{\substack{j}}\psi_{x_j,\delta}(x) \ \hbox{ defined for } \ x \in \Omega(t_0). \]
By properties \partref{harmonic-parabola} and \partref{bends-up}, and by the disjointness of the $B_j$, we can conclude that $u_\delta$ is Lipschitz continuous and superharmonic in $\Omega(t_0)$.    

  Now by property \partref{exterior-ball} (first containment below), property \partref{bends-up} (second containment below), and \eref{positive-speed-outcome2} (third containment below), applied for each $j$
\begin{equation}\label{e.strict-outward-motion}
\bigcup_{j} \tilde{B}_j\cap \Omega(u(t_0)) \subset \Omega(u(t_0)) \setminus \Omega(u_\delta) \subset \bigcup_{j} B_j \subset \R^d \setminus \Omega(u(t_{-1})).
\end{equation}
Thus we conclude that
\[\Omega(u(t_{-1})) \subset \Omega(u_\delta) \subset \Omega(u(t_0))\]
  and also, by \eref{measure-diff-lower-bound},
\[
|\Omega(u(t_0)) \setminus \Omega(u_\delta)| \geq \sum_j |\tilde{B}_j \cap \Omega(u(t_0))| \geq c\eta^{-2d} \mathcal{H}^{d-1}(\Gamma^+(u,t_0)) \delta. 
\]
The previous two displayed equations establish \eref{perturbation-properties}, the first claim of the Lemma.

To conclude we still need to establish the inequality on the difference of the Dirichlet energy \eref{dirichlet-energy-barrier-ineq}. Recall that $u_\delta$ is superharmonic and Lipschitz continuous in $\Omega(t_0)$, so we can apply \lref{intbyparts2} with $v_0 = u_\delta$ and $v_1 = u(t_0)$.  This results in
\begin{align*}
\int_{\Omega(u_\delta)} |\grad u_\delta|^2 \ dx - \int_{\Omega(t_0)} |\grad u(t_0)|^2 \ dx &\leq \int_{\Omega(t_0) \setminus \Omega(u_\delta)}|D{u_\delta}|^2 \ dx\\
&\leq\int_{\Omega(t_0) \setminus \Omega(u_\delta)}(1+\mu_+ - \tfrac{1}{2}\eta_x) \ dx
\end{align*}
where we applied property \partref{gradient-bound} for the second inequality.  This proves \eref{dirichlet-energy-barrier-ineq} and finishes the proof of the Lemma.
\end{proof}

\section{Limit of the minimizing movement scheme} \label{s.energy-soln-existence} 

This section is dedicated to the proof of \tref{existence}. We will study the minimizing movements scheme introduced in \eref{minimizing-movement-scheme}, which we recall here for convenience:
\begin{equation*}\label{e.minimizing-movement-scheme-2}
u_{\delta}^k \in \mathop{\textup{argmin}} \left\{ \mathcal{J}(w) + \Diss(u^{k-1}_\delta, w): w\in F(k\delta)+  H^1_0(U)\right\}.\end{equation*}
Then interpolate discontinuously to define, for all $t \in [0,T]$,
 \[u_{\delta}(t):= u_{\delta}^k \ \hbox{ and } \ F_\delta(t) = F(k\delta) \ \hbox{ if } \ t\in [k\delta, (k+1)\delta).\]
  We will analyze the behavior of the minimizing movements scheme as $\delta \to 0$.

We apply established ideas \cite{alberti2011,MainikMielke} to show that, up to a subsequence, the $u_\delta$ converge pointwise in time to an energy solution \EE{}. Our argument will be quite similar to that in Alberti and DeSimone \cite[Section 4]{alberti2011} which was itself inspired by Mainik and Mielke \cite{MainikMielke}.  The conclusion of these arguments is that, up to a subsequence, the $u_\delta(t)$ and $\Omega(u_\delta(t))$ converge pointwise \emph{everywhere} in time in uniform / Hausdorff distance norm in space to a (global) energetic solution.

\begin{proof}[Proof of \tref{existence}] We start by proving uniform (in $\delta>0$) BV-in-time bounds for the minimizing movements scheme via a discrete Gr\"onwall type argument similar to \lref{energy-soln-time-reg}.  This establishes the necessary compactness to show that, up to a subsequence, the $u_\delta(t)$ and $\Omega(u_\delta(t))$ converge at \emph{every} time in uniform / Hausdorff distance norm to a (global) energetic solution.    The convergence at \emph{every} time, which follows from an application of Helly's selection theorem, is a key detail to establish the energy solution property.  

  The proof will use many of the same ideas that were developed in \sref{basic-energy} in the continuous time case, so we will refer to the relevant arguments above in many places to significantly shorten the presentation.

{\bf Part \partref{existence-BV-bounds}.} A discrete version of the energy dissipation inequality will play a central role in the proof.  We claim that for each $k \geq 1$
\begin{equation}\label{e.one-step-discrete-diss-ineq}
\Diss(u_{\delta}^{k-1}, u_{\delta}^k) \leq \int_{(k-1)\delta}^{k\delta} 2 (1+g_\delta(t))\dot{F}(t)P_{\delta}(t) \ dt  + \mathcal{J}(u^{k-1}_{\delta}) - \mathcal{J}(u^k_{\delta})
\end{equation}
where the error term in the integrand has
\begin{equation}\label{e.gdelta-bd}
|g_\delta(t)| \leq \exp(\|(\log F)'\|_\infty \delta) - 1 \to 0 \ \hbox{ uniformly as $\delta \to 0$,}
\end{equation}
and we used the notation $P_{\delta}(t):= \int _{\partial U } \frac{\partial  u^{[t/\delta]}_{\delta} }{\partial n} dS$.

To prove this  by using the minimizing scheme we build a comparable version of $u^{k-1}_{\delta}$ to $u^k _{\delta}$, by defining $\tilde{u}:= \frac{F(k\delta)}{F((k-1)\delta)} u^{k-1}_{\delta}$. Then we have 
\[
\Diss(u_{\delta}^{k-1}, u^k_{\delta}) + \mathcal{J}(u^k_{\delta}) \leq \Diss (u_{\delta}^{k-1}, \tilde{u}) +\mathcal{J}(\tilde{u})=\mathcal{J}(\tilde{u}) = \mathcal{J}(\tilde{u}) - \mathcal{J}(u_\delta^{k-1}) + \mathcal{J}(u_\delta^{k-1}). 
\]
By integration by parts we can rewrite
\begin{align*}
\mathcal{J}(\tilde{u}) - \mathcal{J}(u^{k-1}_{\delta}) &= \frac{F_k^2 - F_{k-1}^2}{F_{k-1}}\int_{\partial U } \frac{\partial  u^{k-1}_{\delta}}{\partial n} dS\\ &= \int_{(k-1)\delta}^{k\delta} 2 \frac{F(t)}{F_{k-1}} \dot{F}(t)P_{\delta}(t) \ dt.
\end{align*}
Then define $g_\delta(t):= \frac{F(t)}{F_{k-1}} - 1$ for $[t] = k-1$ which satisfies \eref{gdelta-bd} by fundamental theorem of calculus and the inequality $1-e^{-k} \leq e^k-1$ for $k>0$.  That completes our proof of \eref{one-step-discrete-diss-ineq}.

Next we sum up the one-step dissipation inequality \eref{one-step-discrete-diss-ineq} to get, for any $0 \leq t_0 \leq t_1 \leq T$,
\begin{equation}\label{e.full-discrete-diss-ineq}
\Dissbar(u_\delta(t);[t_0,t_1]) \leq \int_{t_0}^{t_1} 2(1+g_\delta(t))\dot{F}(t) P_{\delta}(t) dt  + \mathcal{J}(u_\delta(t_1)) - \mathcal{J}( u_\delta(t_0)) + C\delta
\end{equation} 
where $C\delta$ accounts for the error from $t_0$ and $t_1$ not necessarily being integer multiples of $\delta$ and we are using the uniform Lipschitz bound of $u_\delta(t)$ from \lref{initial-fb-regularity}.

From \eref{full-discrete-diss-ineq} we can derive uniform $BV([0,T];L^1(U) \times \R \times \R)$ bounds on the map $t \mapsto (\chi_{\Omega_{\delta}}(t),P_\delta(t),\mathcal{J}_\delta(t))$ by the same arguments as in \lref{energy-soln-time-reg}.

{\bf Part \partref{existence-helly}.} Due to the above compactness, Helly's selection principle \cite[Theorem 5.1]{alberti2011} yields that, along a subsequence, 
\[(\chi_{\Omega_{\delta}}(t),P_\delta(t),\mathcal{J}_\delta(t)) \to (\chi_{\Omega(t)},P(t),\mathcal{J}(t)) \ \hbox{ in } \ L^1(U) \times \R \times \R \ \hbox{ for \emph{every}} \ t \in [0,T].\] Then arguing as in the proof of part \partref{lr-limits} of \pref{lr-limits} we can also show that $u_\delta(t) \to u(t)$ uniformly in $U$ for each $t \in [0,T]$ and we can show the consistency of the limits 
\[\Omega(t) = \Omega(u(t)), \ \mathcal{J}(t) = \mathcal{J}(u(t)), \ \hbox{ and } \ P(t) = P(u(t)).\]
  Furthermore the convergence of $\Omega_\delta(t) \to \Omega(t)$ can be upgraded to Hausdorff distance convergence using the uniform convergence of the $u_\delta(t) \to u(t)$ and the uniform non-degeneracy of $u_\delta(t)$.

{\bf Part \partref{existence-solution-props}.} Now let us show that $u(t)$ is an energy solution. To show the stability property, observe that for any $v\in u(t) + H_0^1(U)$ we can take $k= [t/\delta]$ and compare $u_\delta(t) = u_{\delta}^k$ with $\tilde{v}:= \frac{F(\delta k)}{F(t)} v(t)$, so that
\[
\Diss(u_{\delta}^{k-1}, u^k_{\delta}) + \mathcal{J}(u^k_{\delta}) \leq \Diss (u_{\delta}^{k-1}, \tilde{v}) +\mathcal{J}(\tilde{v}).
\]
Note that, due to the possibility of a jump in the limit we do not necessarily know that $\Diss(u_{\delta}^{k-1}, u^k_{\delta})  \to 0$.  So to get everything in terms of $u_\delta(t) = u^k_\delta$ we apply the triangle inequality \lref{diss-triangle-eq}
\begin{align*}
\mathcal{J}(u_\delta(t)) = \mathcal{J}(u^k_{\delta}) &\leq \Diss (u_{\delta}^{k-1}, \tilde{v}) - \Diss(u_{\delta}^{k-1}, u^k_{\delta}) +\mathcal{J}(\tilde{v})\\
&\leq \Diss (u^k_{\delta},\tilde{v})+\mathcal{J}(\tilde{v})\\
&=\Diss (u_\delta(t),\tilde{v})+\mathcal{J}(\tilde{v}).
\end{align*}
Now sending $\delta\to 0$  along the convergent subsequence yields the desired stability inequality.

Last we show the energy dissipation inequality. For any $0 \leq t_0 < t_1 \leq T$ we can apply \eref{full-discrete-diss-ineq} to find
\[
\Diss(u_\delta(t_0),u_\delta(t_1))  \leq \int_{t_0}^{t_1} 2(1+g_\delta(t))\dot{F}(t) P_{\delta}(t) dt  + \mathcal{J}(u_\delta(t_1)) - \mathcal{J}( u_\delta(t_0))+ C\delta.
\]
Using the consistency results from part \partref{existence-helly} above and sending $\delta \to 0$ gives the dissipation inequality for $u$. The convergence of the left hand side just follows from the $L^1$ convergence of the indicator functions, and the convergence of the integral is by dominated convergence theorem.
\end{proof}

\appendix

\section{Technical results}

\subsection{Energy difference quotient formulae}
We give and prove several formulas for the difference of the Dirichlet energy under varying positivity set.  We have not seen these particular formulas before in the literature, but they provide a very convenient way to go from energy minimization to viscosity solution properties.

   \begin{lemma}\label{l.intbyparts2}
    Suppose that $v_0,v_1 \in H^1(U)$ with $v_0 \leq v_1$ and $v_0 = v_1 \geq 0$ on $\partial U$.  Call $\Omega_j = \{v_j>0\} \cap U$.
    \begin{itemize}
    \item If $v_1$ is subharmonic in $\Omega_1$ then
    \begin{equation}\label{e.intbyparts2-subsoln} \int_{\Omega_0} |\grad v_0|^2 \ dx- \int_{\Omega_1} |\grad v_1|^2 \ dx \geq \int_{\Omega_1 \setminus \Omega_0}|\nabla v_1|^2 \ dx.\end{equation}
    \item If $v_0$ is superharmonic in $\Omega_1$ then
    \begin{equation}\label{e.intbyparts2-supersoln} \int_{\Omega_0} |\grad v_0|^2 \ dx - \int_{\Omega_1} |\grad v_1|^2 \ dx \leq  \int_{\Omega_1 \setminus \Omega_0}|\nabla {v_0}|^2 \ dx
    \end{equation}
    \end{itemize}
    \end{lemma}

\begin{proof}
 
For \eref{intbyparts2-subsoln} we start with
\begin{align*}
\int_{\Omega_0} |\grad v_0|^2 \ dx - \int_{\Omega_1} |\grad v_1|^2 \ dx &= \int_{\Omega_0 } |\grad v_0|^2 - |\grad v_1|^2 \ dx + \int_{\Omega_1 \setminus \Omega_0} - |\grad v_1|^2 \ dx
\end{align*}
Then we compute the first term on the right
\begin{align*}
 \int_{\Omega_0} |\grad v_0|^2 - |\grad v_1|^2 \ dx &=  \int_{\Omega_0 } |\grad v_1+\grad (v_0 - v_1)|^2-|\grad v_1|^2  \ dx \\
 &=\int_{\Omega_0 } 2\grad v_1\cdot \grad(v_0 - v_1) + |\grad(v_0-v_1)|^2 \ dx \\
 &\geq \int_{\Omega_0 } 2\grad v_1\cdot \grad(v_0 - v_1) \ dx\\
 &= \int_{\Omega_1} 2\grad v_1\cdot \grad ((v_0)_+ - v_1) \ dx +\int_{\Omega_1 \setminus \Omega_0} 2 |\grad v_1|^2 \ dx\\
 &\geq \int_{\Omega_1 \setminus \Omega_0} 2 |\grad v_1|^2 \ dx.
\end{align*}
The last inequality in the sequence above is is due to $v_1$ being $H^1$ subharmonic in $\Omega_1$. Note that $v_1 - (v_0)_+$ is a valid test function for the subharmonicity condition since it is non-negative in $\Omega_1$ and in $H^1_0(\Omega_1)$.

For \eref{intbyparts2-supersoln} we start with
\begin{align*}
\int_{\Omega_0} |\grad v_0|^2 \ dx - \int_{\Omega_1} |\grad v_1|^2 \ dx &= \int_{\Omega_1 } |\grad v_0|^2 - |\grad v_1|^2 \ dx + \int_{\Omega_1 \setminus \Omega_0} - |\grad v_0|^2 \ dx
\end{align*}
Then we compute the first term on the right
\begin{align*}
 \int_{\Omega_1} |\grad v_0|^2 - |\grad v_1|^2 \ dx &=  \int_{\Omega_1 } |\grad v_0|^2-|\grad v_0+\grad (v_1-v_0)|^2  \ dx \\
 &=\int_{\Omega_1 } -2\grad v_0\cdot \grad(v_1 - v_0) - |\grad(v_1-v_0)|^2 \ dx \\
 &\leq \int_{\Omega_1} -2\grad v_0\cdot \grad(v_1 - (v_0)_+) \ dx +\int_{\Omega_1 \setminus \Omega_0} 2 |\grad v_0|^2 \ dx\\
 &\leq \int_{\Omega_1 \setminus \Omega_0} 2 |\grad v_0|^2 \ dx.
\end{align*}
The last inequality in the sequence above is due to $v_0$ being $H^1$ superharmonic in $\Omega_1$. Note that $v_1 - (v_0)_+$ is a valid test function for the superharmonicity condition since it is  in $H^1_0(\Omega_1)$ and is non-negative in $\Omega_1$.

In the smooth case one can also check the related identities
\begin{align*}
 &\int_{\Omega_0} |\grad v_0|^2 \ dx- \int_{\Omega_1} |\grad v_1|^2 \ dx \\
 & \quad =  \int_{\Omega_1} 2\Delta v_1 (v_1 - (v_0)_+) \ dx +  \int_{\Omega_0}  |\nabla(v_1 - v_0)|^2 \ dx + \int_{\Omega_1 \setminus \Omega_0}|\nabla v_1|^2 \ dx
 \end{align*}
and
 \begin{align*} 
 &\int_{\Omega_0} |\grad v_0|^2 \ dx - \int_{\Omega_1} |\grad v_1|^2 \ dx =  \\
 &\quad \int_{\Omega_1} 2(v_1 - (v_0)_+)\Delta v_0  \ dx-\int_{\Omega_1}  |\grad(v_1 - v_0)|^2 \ dx+\int_{\Omega_1 \setminus \Omega_0}|\grad{v_0}|^2 \ dx.
  \end{align*}

\end{proof}

\subsection{Sharp triangle inequality for the dissipation distance}
Here we give a sharp triangle inequality for the dissipation distance.  The triangle inequality is important for most of the standard theory of rate-independent energetic evolutions, although we remark that it does not seem to generalize well to the general anisotropic case.  The sharp triangle inequality is used in \pref{lr-limits} to show that if an energetic solution jumps multiple times then $u(t)$ must always be in between $u_\ell(t)$ and $u_r(t)$.
\begin{lemma}[Dissipation distance sharp triangle inequality]\label{l.diss-triangle-eq}
Let $A$, $B$, and $C$ be arbitrary finite measure regions in $\R^n$. Then
\[\textup{Diss}(A,B) + \textup{Diss}(B,C) - \textup{Diss}(A,C) = (\mu_-+\mu_+)\bigg[|B \setminus (A\cup C)| + |(A \cap C) \setminus B|\bigg]\]
\end{lemma}
\begin{figure}[t]
\begin{tikzpicture}[thin,scale = 0.4]

\def\radiusBS{1}
\def\radiusA{2}
\def\radiusBL{5}
\def\radiusC{3}
\coordinate (A) at (0,0);
\coordinate (BS) at (0,1.5);
\coordinate (BL) at (4,0);
\coordinate (C) at (0,0);


\begin{scope}
    \fill[gray!50] (0,0) ellipse (\radiusBL cm and \radiusBS cm);
    \clip (0,0) ellipse (\radiusC cm and \radiusC cm);
    \fill[white] (0,0) ellipse (\radiusC cm and \radiusC cm);
\end{scope}

\begin{scope}
    \fill[gray!10] (0,0) ellipse (\radiusC cm and \radiusC cm);
    \clip (0,0) ellipse (\radiusA cm and \radiusA cm);
    \fill[white] (0,0) ellipse (\radiusC cm and \radiusC cm);
\end{scope}

\begin{scope}
    \fill[gray!30] (0,0) ellipse (\radiusA cm and \radiusA cm);
    \clip (0,0) ellipse (\radiusA cm and \radiusA cm);
    \fill[white] (0,0) ellipse (\radiusBL cm and \radiusBS cm);
\end{scope}

\draw[draw=black] (0,0) ellipse (\radiusC cm and \radiusC cm);
\draw[draw=black] (0,0) ellipse (\radiusA cm and \radiusA cm);
\draw[draw=black] (0,0) ellipse (\radiusBL cm and \radiusBS cm);

\draw[<-,dotted] (BS) -- ++(2,2) node[right]{$(A \cap C) \setminus B$};
\draw[<-,dotted] (BL) -- ++(2,0) node[right]{$B \setminus (A \cup C)$};

\end{tikzpicture}
\caption{Diagram of the sharp triangle inequality for the dissipation distance. Here $A \cap C$ is the inner circle, $A \cup C$ is the outer circle, and $B$ is the ellipse.}
\label{f.set-theory}
\end{figure}
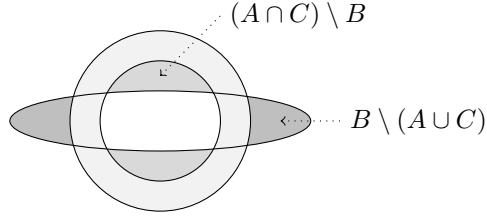
\begin{proof}
The result is purely set algebraic computations, see \fref{set-theory} for the geometric idea. We compute
\begin{align}
&\textup{Diss}(A,B)+\textup{Diss}(B,C) - \textup{Diss}(A,C) \notag\\
&= \int \mu_+{\bf 1}_{B \setminus A} + \mu_-{\bf 1}_{A \setminus B}+\mu_+{\bf 1}_{C \setminus B}+\mu_-{\bf 1}_{B \setminus C} - \mu_+{\bf 1}_{C \setminus A} - \mu_-{\bf 1}_{C \setminus A} \ dx\notag\\
&= \int \mu_+[{\bf 1}_{B \setminus A} +{\bf 1}_{C \setminus B}- {\bf 1}_{C \setminus A}]+ \mu_-[{\bf 1}_{A \setminus B}+{\bf 1}_{B \setminus C} - {\bf 1}_{A \setminus C}]  \ dx. \label{e.expanded-triangle-id}
\end{align}
Now we just compute by force the first term in brackets on the previous line
\begin{align*}
{\bf 1}_{B \setminus A} +{\bf 1}_{C \setminus B}- {\bf 1}_{C \setminus A} & = {\bf 1}_{B} - {\bf 1}_{A \cap B}+{\bf 1}_{C}- {\bf 1}_{ B \cap C}- [{\bf 1}_{C} - {\bf 1}_{ A \cap C}]\\
&  ={\bf 1}_{B}+{\bf 1}_{ A \cap C} - {\bf 1}_{ B \cap C} - {\bf 1}_{ B \cap A}\\
& = {\bf 1}_{B}+{\bf 1}_{ A \cap C}{\bf 1}_{B} - {\bf 1}_{ B \cap C} - {\bf 1}_{ B \cap A} +{\bf 1}_{ A \cap C}(1-{\bf 1}_{B}) \\
&  = {\bf 1}_{B}(1+{\bf 1}_{ A \cap C} - {\bf 1}_{ C} - {\bf 1}_{  A}) + {\bf 1}_{ (A \cap C) \setminus B}\\
& = {\bf 1}_{B}(1-{\bf 1}_{A \cup C})+ {\bf 1}_{ (A \cap C) \setminus B}\\
& = {\bf 1}_{B\setminus (A \cup C)}+ {\bf 1}_{ (A \cap C) \setminus B}.
\end{align*}
The second bracketed term in \eref{expanded-triangle-id} works out quite similarly and is identical after the second step above giving the same result.

Plugging these identities back into \eref{expanded-triangle-id} gives the result.
\end{proof}

\bibliography{monotone-articles}

\end{document}